\documentclass[11pt]{article}
\usepackage{ifthen}
\usepackage{algorithm}
\usepackage{algpseudocode}
\newboolean{arxiv}
\setboolean{arxiv}{true}

\usepackage{lipsum}

\usepackage{ifthen}

\ifthenelse{\boolean{arxiv}}{
\usepackage{arxiv_def}
}
{
\usepackage{times}
\usepackage{colt_def}
}

\usepackage[toc,page]{appendix}
\usepackage{xr}
\usepackage{mathtools}
\usepackage{yfonts}   
\usepackage{subcaption}

\DeclareSymbolFont{Greekletters}{OT1}{iwona}{m}{n}
\DeclareSymbolFont{greekletters}{OML}{iwona}{m}{it}
\DeclareMathSymbol{\salpha}{\mathord}{greekletters}{"0B}
\DeclareMathSymbol{\sbeta}{\mathord}{greekletters}{"0C}
\DeclareMathSymbol{\sgamma}{\mathord}{greekletters}{"0D}
\DeclareMathSymbol{\sOmega}{\mathord}{Greekletters}{"0A}
\DeclareMathSymbol{\svarepsilon}{\mathord}{greekletters}{"22}
\DeclareMathSymbol{\svarrho}{\mathord}{greekletters}{"25}
\DeclareMathSymbol{\svarphi}{\mathord}{greekletters}{"27}

\makeatletter
\newcommand{\vast}{\bBigg@{3}}
\newcommand{\Vast}{\bBigg@{4}}
\makeatother

\begin{document}

\title{Topological trivialization in\\ non-convex empirical risk minimization}

\author{
Andrea Montanari\thanks{Department of Statistics and Department of Mathematics, Stanford University} 
	\and 
	Basil Saeed\thanks{Department of Electrical Engineering, Stanford University}
}

\maketitle

\begin{abstract}
Given data $\{(\bx_i,y_i): i\le n\}$,  with 
$\bx_i$ standard $d$-dimensional Gaussian feature vectors, and $y_i\in\reals$
response variables, we study the general problem of
learning a  model parametrized by $\btheta\in\R^d$, by minimizing a
loss function that depends on $\btheta$ via the one-dimensional projections $\btheta^\sT\bx_i$. While previous work mostly dealt with convex losses, our
approach assumes general (non-convex) losses hence covering classical, yet poorly understood examples such  as the perceptron and non-convex robust regression.

We use the Kac-Rice formula to control the asymptotics of the expected number of 
local minima of the empirical risk, under 
the proportional asymptotics $n,d\to\infty$, $n/d\to\alpha >1$. 
Specifically, we prove a finite dimensional variational formula 
for the exponential growth rate of the expected number of local minima. 
Further we provide sufficient conditions under which the exponential 
growth rate vanishes and all empirical risk minimizers
have the same asymptotic properties (in fact, we expect the minimizer 
to be unique in these circumstances).
We refer to this phenomenon as `rate trivialization.'

If the population risk has a unique minimizer, our sufficient condition for rate trivialization is typically 
verified when  
the samples/parameters ratio $\alpha$ is larger than
a suitable constant $\alpha_{\star}$. Previous general results of this type 
required $n\ge Cd \log d$.

We illustrate our results in the case of non-convex robust regression.
Based on heuristic arguments and numerical simulations, we present
a conjecture for the exact location of the trivialization phase transition
$\alpha_{\str}$.
\end{abstract}

\section{Introduction}
Given $n$  feature vector-label pairs $(\bx_i,y_i)\in\R^{d}\times \R$, $i\in[n]$,
we learn a model $\bx\mapsto \varphi(\btheta^{\sT}\bx)$  by 
Empirical Risk Minimization (ERM) as follows:
\begin{equation}
\label{eq:ERM_0}
\hbtheta = \arg\min_{\btheta \in\R^d} \hat R_n(\btheta),\quad\quad
     \hR_n(\btheta) := \frac1n \sum_{i=1}^n L(\varphi(\btheta^\sT\bx_i),y_i) + \frac{\lambda}{2}\|\btheta\|_2^2.
\end{equation}
Here,  $L: \R^{2} \to \R$ is a loss
function, $\lambda\ge 0$ is the (ridge) regularization parameter,
and the activation function $\varphi:\R\to\R$ is assumed to be given.
We will be interested in the structure of the empirical risk landscape
under the high-dimensional asymptotics $n,d\to\infty$, $n/d\to\alpha$

Over the last decade, the study of ERM in the high-dimensional  regime $n\asymp d$
has provided insights into a wealth of phenomena: phase transitions
\cite{DMM09,BayatiMontanariLASSO,lelarge2019fundamental}; 
information-computation gaps
\cite{celentano2022fundamental,schramm2022computational};  
benign overfitting and double-descent \cite{hastie2022surprises},
to name a few.
Most of this theoretical work has been restricted to the convex case.
While convex or linear problems provide good approximations to non-convex settings in certain regimes, they remain rough analogies, and fail to capture several important aspects of truly non-convex models.

One important general insight regarding the empirical risk (train loss) $\hR_n(\btheta)$
is that, for large enough $n/d$ it should concentrate around the expected risk (test loss) $R(\btheta)=\E \hR_n(\btheta)$, uniformly over $\|\btheta\|_2\le D_{\max}$~\cite{van1998asymptotic}. 
This observation can be generalized to the concentration of the gradient 
and Hessian processes $\nabla\hR_n(\btheta)$, $\nabla^2\hR_n(\btheta)$.
This allows showing that certain topological properties of the empirical risk landscape
match the ones of the expected risk for large enough $n/d$ \cite{mei2018landscape}.
For instance (under technical conditions) 
\cite{mei2018landscape} proves that if the test loss 
$R(\btheta)$ has isolated strongly convex local minima, then 
the train loss $\hR_n(\btheta)$  has isolated strongly convex local minima in a neighborhood of those of $R(\btheta)$, provided $n/d\ge C\log d$.
However, the uniform convergence approach fails
to provide a sharp characterization of local minima in
the proportional asymptotics $n\asymp d$. In particular, it provides  only upper bounds on the generalization
and estimation error and not precise asymptotics.

In this work, we build on the Kac-Rice approach recently developed 
in~\cite{asgari2025local} to obtain sharp characterizations of
a class of non-convex ERM  under the proportional asymptotics $n/d\to\alpha$.

Generalizing its usage in spin glass theory  \cite{fyodorov2012critical}, we will use the term 
`topological trivialization' to cases in which (with high probability)
the first order stationary points of $\hR_n$ are in one-to-one correspondence with 
the ones of $R$, with matching index.
%

In the present paper we present the following results:
\begin{enumerate}
\item[{\sf R1.}] A variational formula for the exponential growth rate of the expected number of
local minima of the empirical risk $\hR_n$.
\item[{\sf R2.}] Sufficient conditions for this exponential growth rate to become non-positive.
This enables to show that, for $n/d$ above a constant, a weak form of topological
trivialization (which we call `rate trivialization') holds. Whenever the above conditions
hold, we obtain a sharp asymptotic characterization of the local minima.
\item[{\sf R3.}] We specialize these results to non-convex M-estimation and 
as a case study we consider robust regression with Tukey loss. 
We demonstrate how our asymptotic characterization can be accurately evaluated,
and are predictive of numerical experiments.
\end{enumerate}

Well established approaches for analyzing ERM under proportional asymptotics
are ill-suited to address the above questions in non-convex settings.
A particularly elegant technique is based on Gordon min-max comparison inequality~\cite{thrampoulidis2014gaussian,thrampoulidis2018precise,miolane2021distribution}. 
When the composition $t\mapsto L(\varphi(t),y)$ is non-convex, 
this approach only guarantees a lower bound on the minimum train loss,
and this lower bound is generally not asymptotically tight.

An alternative approach to analyzing high-dimensional ERM problems is the use of Approximate Message Passing (AMP) algorithms
\cite{BayatiMontanariLASSO,donoho2016high}.
Here, one designs an iterative algorithm $\hat\btheta_n^\up{t}$ to minimize $\hat R_n(\btheta)$, and derives an
asymptotically exact characterization the sequence of iterates 
---known as `state evolution'--- for $t\in\{0,1,\dots, T\}$, with $T=O(1)$ as $n,d\to\infty$. 
Then one shows that the $t \to\infty$ of state evolution describes the $n,d\to\infty$
asymptotics of the minimizers of $\hat R_n(\btheta)$. 
In absence of convexity, designing an AMP algorithm that provably converges to every 
minimizer of interest can be  challenging.

A technique that combines Gaussian comparison inequalities and AMP analysis
was successful in characterizing the global minimum of the empirical risk for
non-convex matrix factorization problems \cite{montanari2015non} and, recently, 
supervised learning problems of the type studied in the present paper \cite{vilucchio2025asymptotics}.

As mentioned above, in this paper we overcome limitations of earlier approaches by adopting the Kac-Rice method of~\cite{asgari2025local}. While this approach presents its own set of technical challenges,
it has the important advantage of giving access to a quite detailed picture of the empirical
risk landscape. In particular it gives access to a quite complete picture of the
empirical risk landscape. When this characterization is sufficiently benign, this in particular
implies convergence of a variety of standard optimization algorithms.

In the rest of the introduction, we  review the
Kac-Rice approach and describe more precisely the main contributions of this paper. 
We refer to \cite{fan2021tap,celentano2023local}.

\paragraph{Localizing ERM minimizers via Kac-Rice.} We assume data $\{(\bx_i,y_i):i\le n\}$ to be i.i.d.
with $\bx_i\sim\normal(\bzero,\bSigma)$ and $y_i$ depending on $\bx_i$ only via the one-dimensional projection $\btheta_0^{\sT}\bx_i$.

Let $\cuP(\R^{m})$ denote the set of probability distributions on $m$-dimensional vectors.
For a given $\btheta\in\R^{d}$, define the empirical distributions
\begin{equation}
    \hmu(\btheta):= \frac1d\sum_{j=1}^d \delta_{\sqrt{d} 
    (\bSigma^{1/2}[\btheta,\btheta_0])_j} \in\cuP(\R^{2})
    \quad\quad
    \hnu(\btheta):= \frac1n\sum_{i=1}^n \delta_{\btheta^\sT\bx_i, \btheta_0^\sT\bx_i, w_i} \in\cuP(\R^{3}).
\end{equation}
Here,
$(\bSigma^{1/2}[\btheta,\btheta_0])_j$ denotes the $j$th row of the $d\times 2$ dimensional matrix $\bSigma^{1/2}[\btheta,\btheta_0].$
For some sufficiently regular subsets of probability distributions $\cuA\subseteq \cuP(\R^{2}),\cuB\subseteq\cuP(\R^{2})$, 
let
\begin{align}
\nonumber
\cZ_n(\cuA, \cuB):=\Big\{\mbox{ local minimizers of }  \hR_n(\btheta)
\mbox{ s.t.  }  (\hmu(\btheta),\hnu(\btheta)) \in\cuA \times \cuB \Big\}\, 
\end{align}
and assume there exists a lower semicontinuous $\Phi:\cuP(\R^{2})\times \cuP(\R^3)\to\R$ such that
\begin{align}
\label{eq:markov_localization}
\P\Big(  \cZ_n(\cuA, \cuB)>0
\Big)
&\le
\exp\left\{-n\inf_{(\mu,\nu)\in(\cuA\times \cuB)}
\Phi(\mu,\nu)+o(n)\right\}.
\end{align}
We call such a $\Phi$ a \emph{valid lower rate function} and, 
following \cite{asgari2025local}, we say that the 
\emph{rate trivialization property} holds if
there exists a valid lower rate function and $(\mu_\star,\nu_\star)$ such that:
\begin{equation}\label{eq:Trivialization}
\begin{split}
    &\Phi(\mu,\nu)\ge 0 \;\;\;\forall \mu,\nu\, ,\\
    &\Phi(\mu,\nu) = 0 \;\;\Leftrightarrow  (\mu,\nu) = (\mu_\star,\nu_\star).
    \end{split}
\end{equation}
Informally, local minimizers $\hbtheta$ are exponentially unlikely to have empirical distribution 
$\hmu(\hbtheta)$, $\hnu(\hbtheta)$ in regions $\cuA\times\cuB$ that are 
separated from $(\mu_*,\nu_*)$.
In~\cite{asgari2025local} (see also \cite{maillard2020landscape}) 
a valid lower rate function is computed  by bounding
$\E\left[ |\cZ_n(\cuA, \cuB)|\right]$:
\begin{align}
    \lim_{n,d\to\infty}\frac{1}{n}\log \E\left[ |\cZ_n(\cuA, \cuB)|\right]
    \le -\inf_{(\mu,\nu)\in\cuA \times\cuB}\Phi(\mu,\nu)\, .\label{eq:Summary}
\end{align}
Such a $\Phi$ is valid by Markov's inequality.

Under trivialization, taking first $\cuA\times \cuB = \Ball^c_{W_2}(\mu_\star,\eps)\times\cuP(\R^2)$ and then $\cuA\times \cuB = \cuP(\R^2)\times 
\Ball^c_{W_2}(\nu_\star,\eps)$ in Eq.~\eqref{eq:markov_localization}
(with $\Ball^c_{W_2}(\gamma_\star,\eps)$ the complement of 
the $\eps$ balls in the $W_2$ topology around $\gamma_\star$)
allows us to conclude that, for any local minimizer $\hat\btheta$,
\begin{align}
\nonumber
\hmu(\hat\btheta)\stackrel{W_2}{\Rightarrow}\mu_\star\,,\;\;\;\;\;\;
\hnu(\hat\btheta)\stackrel{W_2}{\Rightarrow}\nu_\star\, .
\end{align}
The measures $\mu_\star,\nu_\star$ can then be used to give a description of the asymptotics of $\hat \btheta,$ such as the estimation error or training error at $\hat\btheta$. 

 \begin{figure}[h!] 
        \centering 
        \begin{subfigure}{0.48\textwidth} 
            \centering
            \includegraphics[width=\textwidth]{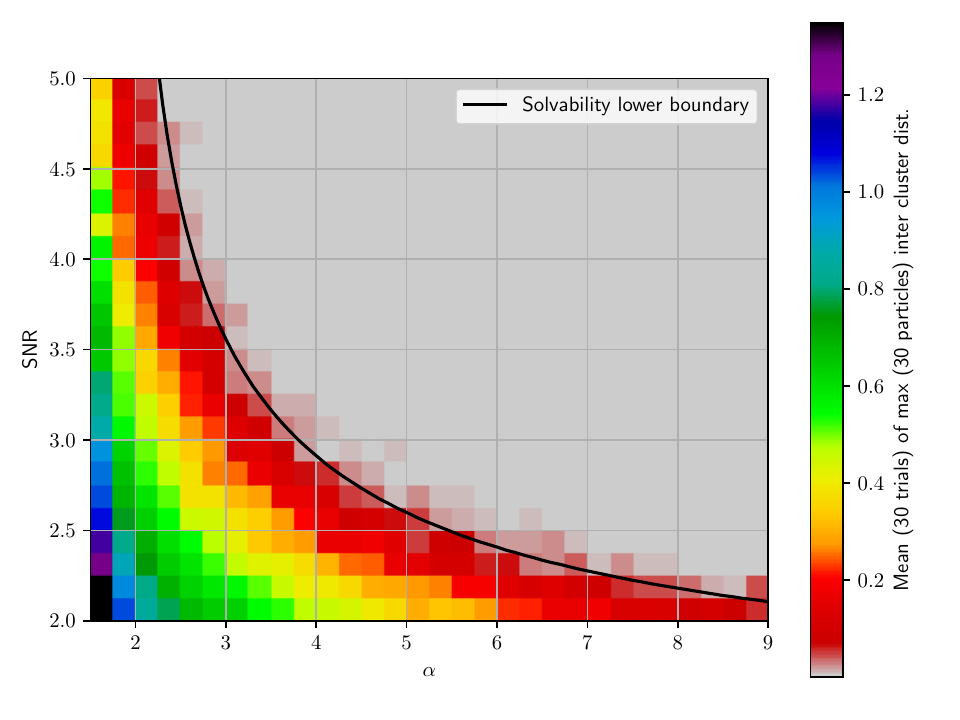} 
            \caption{}
            \label{fig:PhaseDiagram_1}
        \end{subfigure}%
        \hfill 
        \begin{subfigure}{0.48\textwidth} 
            \centering
            \includegraphics[width=\textwidth]{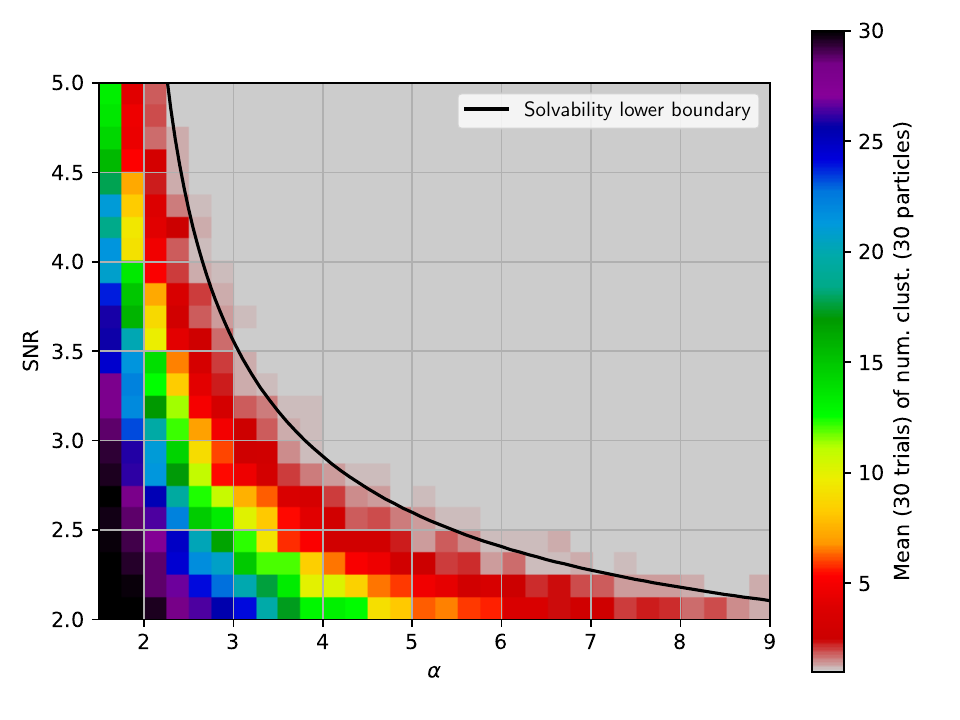} %
            \caption{}
            \label{fig:PhaseDiagram_2}
        \end{subfigure}
        \caption{ 
            Tukey regression:
        Output of gradient descent from $M=30$ random 
        initializations, on the same empirical risk landscape (same data 
        $\{y_i,\bx_i\}_{i\le n}$), to produce
         estimates  $\hat\btheta^\up{1},\dots,\hat\btheta^\up{30}$.
        Here $d=200$ and results are averaged over $N=30$ different 
         Left:
        Maximum distance $\max_{i,j}\|\hat\btheta^\up{i}-\hat\btheta^\up{j}\|_2$ 
        as a function of $\alpha$ and the $\SNR.$ 
        Right: Number of clusters  formed by the point estimates  
        $\hat\btheta^\up{1},\dots,\hat\btheta^\up{30}$
        (clusters are constructed by thresholding the normalized distance
        $\|\hat\btheta^\up{i} - \hat\btheta^\up{j}\|/(\|\hat\btheta^\up{i}\|_2 \|\hat\btheta^\up{j}\|_2)^{1/2}$
        at $\eps = 10^{-3}$).
        }
        \label{fig:PhaseDiagram}
    \end{figure}

\paragraph{Main results}
The results of this paper are as follows.
\begin{enumerate}
    \item[{\sf R1.}] For the functional $\Phi(\mu,\nu)$ of \cite{asgari2025local},
    describing the exponential growth rate of the expected number of local minima,
     we prove 
    a simple characterization of the infinite-dimensional variational formula 
    over $(\mu,\nu)$ appearing in Eq.~\eqref{eq:Summary}. 
     When $\cuA,\cuB$ are defined in terms of linear constraints on $\mu,\nu$, we show that this formula reduces to a finite-dimensional one involving only a fixed number of parameters
     (Theorem~\ref{thm:fin_dim_var_formula_k=1}).
    \item[{\sf R2.}] By analyzing this simplified characterization, we give conditions on under which the trivialization property of Eq.~\eqref{eq:Trivialization} holds.
    Namely, we give an explicit formula for a constant $\alpha_{\star}$
    (depending on the loss and target distribution)
    such that rate trivialization takes place
    for over-sampling ratio above this threshold $\alpha >
    \alpha_\star$. We also derive the analytic form of $\mu_\star,\nu_\star$ at which trivialization occurs (Theorem~\ref{thm:trivialization_k=1}). These analytic formulas  directly generalize the ones found in the literature on high-dimensional asymptotics of convex ERM problems.
    \item[{\sf R3.}] We apply our results to study non-convex M-estimation (Theorem~\ref{cor:robust_regression}) and  in particular, 
    robust regression with Tukey loss.
    We specialize our analytic formula to this setting, demonstrating that
    it can be accurately evaluated.
    We further compare predictions with numerical simulations, showing excellent agreement 
    for moderately large $n,d$ and $\alpha=n/d$ of order one.

    In fact, we demonstrate numerically that our predictions are accurate for a 
    larger range $\alpha\in (\alpha_{\str},\infty)$ than predicted by the theory.
    Using heuristic arguments, we derive a precise conjecture for the 
    trivialization threshold $\alpha_{\str}$
    (Section~\ref{sec:tukey}).
\end{enumerate}

As a preview, Fig.~\ref{fig:PhaseDiagram} displays numerical results for 
Tukey regression. We demonstrate the impact of the trivialization
phase transition at $\alpha_{\str}$ on the gradient descent (GD) 
dynamics. Namely, we run GD from a random initialization on the unit sphere,
repeating the experiment $30$ times for each realization of the data $\bX,\by$, for a total of $30$ realizations.
We thus obtain $30$ GD outputs $\hbtheta^{(1)}$, \dots $\hbtheta^{(30)}$ for 
each empirical risk landscape.
We report the average over the landscape realization of two geometric properties of this high-dimensional point cloud
in the plane $(\alpha,\SNR)$, with $\SNR := \|\btheta_0\|_2/\Var(w)$ the signal-to-noise ratio.
On the \emph{left}, we  report the maximum distance between any two outputs 
$\max_{i,j}\|\hat\btheta^\up{i}-\hat\btheta^\up{j}\|_2$: this allows to probe whether GD converges to a unique minimizer
regardless of the initialization. On the \emph{right}, we plot the number of clusters obtained by 
clustering the GD outputs $\hbtheta^{(1)}$, \dots $\hbtheta^{(30)}$. We refer to Section \ref{sec:Numerical}
and Appendix \ref{sec:numerical_details} for further details about these numerical experiments.

Results are consistent with a sharp change at the trivialization
threshold $\alpha_{\str}$. For $\alpha>\alpha_{\str}$, GD converges to the same local minimum with high probability with respect to the initialization; for $\alpha<\alpha_{\str}$, different random runs result in very different outputs.

\section{Main results}
As mentioned in the introduction, we assume that, for every $i\in[n]$, and every 
Borel set $\cB\subseteq \R$,
    $\P\left(y_i \in \cB | \bx_i \right) = \mathsf{P}(\cB| \btheta_0^\sT\bx_i).$
Equivalently, for some random vector $\bw\in\R^n$ independent of $(\bx_i)_{i\le n}$, and $h:\R^2 \to\R$, we can write 
\begin{equation}\label{eq:xi_yi}
    y_i = h(\btheta_0^\sT\bx_i, w_i).
\end{equation}
Mathematically, we can then rewrite the empirical risk minimization problem in~\eqref{eq:ERM_0} as
\begin{equation}
\label{eq:erm_obj}
    \hat R_n(\btheta) = \frac1n \sum_{i=1}^n \ell(\btheta^\sT \bx_i; \btheta_0{^\sT} \bx_i, w_i) + \lambda \|\btheta\|_2^2,
\end{equation}
where we introduced $\ell : \R^{3} \to \R$. Even though $\ell$ is obtained by the composition of $L$ and $\varphi$ (cf.~\eqref{eq:ERM_0}), we will often refer to $\ell$ as the `loss.'
Of course the expression \eqref{eq:erm_obj} is only relevant for theoretical analysis, and
for actual estimation we minimize the risk \eqref{eq:ERM_0}.

For the rest of the paper, we will continue with this formulation and present our assumptions and results using this notation. Throughout, we will denote by $\ell'(v;v_0,w)$, $\ell''(v;v_0,w)$ the derivatives of $\ell$ with respect to its
first argument.
\subsection{Assumptions}
\label{sec:assumptions}
\begin{assumption}[Regime]
\label{ass:regime}
We assume the proportional asymptotics, i.e., $d := d_n$ such that
   \begin{equation}
       \alpha_n := \frac{n}{d_n} \to \alpha \in (0, \infty).
   \end{equation}
\end{assumption}

\begin{assumption}[Loss]
\label{ass:loss}
\label{ass:density}
The following partial derivatives exist and are Lipschitz continuous:
\begin{equation}
\nonumber
    \frac{\partial}{\partial u_1}\ell(u_1; u_0,w),\quad\quad
    \frac{\partial^2}{\partial {u_1} \partial{u_i}} \ell(u_1;u_0,w),\quad\quad
    \frac{\partial^3}{\partial u_1 \partial{u_i}  \partial u_j}\ell(u_1;u_0,w), \quad\quad\textrm{for}\quad\quad i,j\in\{0,1\}.
\end{equation}

In particular, this implies that
\begin{equation}
   \sfL_{\min} :=   -\Big( 0 \wedge \inf_{\substack{ (v,v_0) \in \R\times\R \\w \in\R}}  \ell''(v;v_0,w)  \Big) < \infty
\quad\quad
   \sfL_{\max} :=   \sup_{\substack{ (v,v_0) \in \R\times\R \\w \in\R}}  
   \ell''(v;v_0,w) < \infty\, .
\end{equation}
\end{assumption}

\begin{assumption}[Data distribution]\label{ass:Data}
The covariates $(\bx_i)_{i\in [n]}$ are i.i.d. with $\bx_i\sim\cN(\bzero,\bSigma)$ for $\bSigma$ with condition number $O(1)$.
If $\lambda \neq 0$, we will assume $\bSigma = \bI_d$.
The noise vector $\bw = (w_i)_{i\in[n]}$ is independent of $(\bx_i)_{i\in [n]}$ and satisfies,
almost surely,
\begin{equation}
   \frac1n \sum_{i=1}^n \delta_{w_i} \stackrel{W_2}{\Rightarrow} \P_w
\end{equation}
for some $\P_w \in\cuP(\R)$  with finite second moment.
\end{assumption}

\begin{assumption}[Empirical distribution of the entries of $\btheta_0$]
\label{ass:theta_0}
There exists a constant $c>0$ independent of $n$ such that 
$\|\bSigma^{1/2}\btheta_0\|_2 > c $.
   Furthermore, for some $\mu_0 \in\cuP(\R)$, we have the convergence 
   \begin{equation}
\nonumber
       \frac1d \sum_{j=1}^d \delta_{(\bSigma^{1/2}\btheta_{0})_j}\stackrel{W_2}{\Rightarrow} \mu_0,
       \quad\quad\textrm{with}\quad\quad r_{00} := \int \!t_0^2\,\mu_0(\de t_0).
   \end{equation}
\end{assumption}

\subsection{Definitions}
\label{sec:definitions}

\paragraph{Random matrix preliminaries.}
Let $\upper:=\{z\in\complex:\;\Im(z)>0\}$ be the upper half complex plane.
For $z\in\C$ with $\Im(z) >0$,  $\nu\in\cuP({\R^{3}})$,
we let $s_{\star} \equiv s_{\star}(\nu;z)$ be the the unique solution 
in $\upper$ of the equation
\begin{equation}
\label{eq:fp_eq}
 \E_\nu\left[\frac{
\ell''(v;v_0,w)
     }{1 + \ell''(v;v_0,w) s} 
     \right] - z 
= \frac{1}{\alpha s},
\end{equation}
Since we assume that $|\ell''(v;v_0,w)|$ is bounded, the
existence and uniqueness of $s_{\star}(\nu;z)$ is a classical result in random matrix theory~\cite[Theorem 1.1 and Lemma 5.1]{silverstein1995empirical}.
Further, $s_{\star}(\nu;z)$ is the Stieltjes transform on a unique measure on $\reals$ which we denote by $\mu_{\star}(\nu)\in\cuP(\R)$,
also known as a generalized Marchenko-Pastur (MP) law.

\paragraph{Functionals of $\btheta$.}
Given a probability measure $\mu\in\cuP(\R^{2})$, we define
\begin{equation}
   \bR(\mu) := 
   \begin{bmatrix}
      r_{11}  & r_{10}\\
      r_{01}  & r_{00}\\
   \end{bmatrix}
   :=
   \int \bt\bt^{\sT} \mu(\de\bt), \quad\quad\textrm{so that}\quad\quad
\bR(\hmu(\btheta)) :=\begin{bmatrix}
\btheta^\sT\bSigma\btheta & \btheta^{\sT}\bSigma\btheta_0\\
\btheta_0^{\sT}\bSigma\btheta & \btheta_0^{\sT}\bSigma\btheta_0\\
\end{bmatrix}.
\end{equation}

Further, for $\bv,\bv_0,\bw\in \R^{n}$, define 
\begin{equation}
\label{eq:def_G}
\bg(\bv,\bv_0,\btheta;\bw) := \frac1n [\bv,\bv_0]^\sT \ell'(\bv;\bv_0,\bw) + \lambda [\btheta,\btheta_0]^\sT\btheta \in \R^{2},
\end{equation}
with $\ell':\R^{3}\to\R$ applied row-wise to $[\bv,\bv_0,w]\in\R^{n\times 3}$.
In particular, notice that $\bg(\bX\btheta, \bX\btheta_0,\btheta;\bw) = 
\bzero$
encodes the equations 
\begin{equation}
\btheta^\sT\grad_{\btheta}\hat R_n(\btheta) =0,\quad\quad
\btheta_{0}^\sT\grad_{\btheta}\hat R_n(\btheta) =0.
\end{equation}

\paragraph{Sets of local minimizers.}
We will consider local minimizers of the ERM problem satisfying a set of constraints. 
Namely, for $\cuA\subseteq \cuP(\R^{2}),\cuB\subseteq\cuP(\R^{3})$, 
we define 
\begin{align}
\label{eq:set_of_zeros_main}
\cZ_n&(\cuA,\cuB):=
\Big\{\btheta\in \R^{d}:\; 
(\hmu(\btheta),\hnu(\btheta))\in \cuA\times \cuB,\;
\nabla \hR_n(\btheta)=\bzero,\; 
\nabla^2 \hR_n(\btheta)\succeq\bzero
\Big\}.
\end{align}
We will further define the set of ``well-behaved'' points as follows:
\begin{align}
\label{eq:well_behaved}
\cG_n(\sfA_{R},\sfa_L) :=
\big\{ \btheta\in\R^d\, : 
\sfA_R^2 \succ \bR(\hmu(\btheta)),\;
&\|\ell'(\bX\btheta;\bX\btheta_0,\bw)\|^2_2 > \sfa_L^2n,\;\\
&\quad\quad\sigma_{\min}\left( \bJ \bg(\bX\btheta,\bX\btheta_0,\btheta) \right) > \, e^{-\overline{o}(n)}
\big\}.\nonumber
\end{align}
where $\overline{o}(n)$ is any fixed function such that $\lim_{n\to\infty}\overline{o}(n)/n = 0$
(but we do not track the dependence on $\overline{o}(\,\cdot\,)$),
$\bJ \bg := \bJ_{(\bv,\bv_0,\btheta)}\bg\in \R^{2\times 2n+d}$ denotes the Jacobian of the map defined in Eq.~\eqref{eq:def_G}, and $\ell'$ is once again applied row-wise.
Note that both $\cZ_n$ and $\cG_n$ are functions of $\bX, \btheta_0$ and $\bw$,
    but we suppress this in the notation.

\paragraph{Notations.}
We use $\Ball^{d}(\bx_0,r)$ for the Euclidean ball of radius $r$ in $\R^d$ centered at
$\bx_0$. We denote by $\sfS_{\succeq\bzero}^{k}$ and $\sfS_{\succ\bzero}^{k}$ 
the sets of $k\times k$ symmetric positive semidefinite and (respectively)
strictly positive definite matrices. Given a probability measure
$\mu$ over variables $x,y,z,\dots$, we denote by $\mu_{\cdot|x}$ the 
conditional distribution given $x$, and by $\mu_{(x)}$ the 
marginal distribution of $x$. Given two probability measures $\mu_1,\mu_2$
with finite second moment, we denote by $W_2$ their Wasserstein-$2$ distance.

\subsection{The finite-dimensional variational formula}
Throughout, when considering probability measures $\mu \in\cuP(\R\times \R)$ and $\nu\in\cuP(\R \times \R \times \R),$
we use $(t,t_0) \in\R \times \R$ and $(v,v_0,w)\in\R\times \R\times \R$
to denote random variables with these distributions, respectively.

The following functionals and sets will play a crucial role in our theorems to be stated below. 
For $\bQ=(q_{ij})_{1\le ij\le 2}\in\reals^{2\times 2}$ we denote by $\schur(\bQ,q_{22})=q_{11}-q_{12}q_{22}^{-1}q_{21}$ the Sch\"ur complement with respect to the designated entry $q_{22}$. 
Let
\begin{align}
\label{eq:phi_A_phi_B_defs}
&\Phi_\mupart(\mu) :=
\frac1\alpha 
\KL\left(\mu_{\cdot|t_0}\| \cN(r_{10}(\mu) r_{00}^{-1} t_0,\schur(\bR,r_{00})\right)\, ,\\
    &\Phi_{\nupart}(\nu,\bR,s)  :=   -\lambda s - \E[\log (1+ s\ell''(v;v_0,w)]   
    +  \frac1{2\alpha}\log \left(
    \frac{
    \alpha\, e \,s^2\, \E_\nu[\ell'(v;v_0,w)^2]}{
\schur(\bR,r_{00})
    }\right) \nonumber
    +\KL(\nu_{\cdot|w} \| \cN(\bzero ,\bR ))\, .
\end{align}
Further, let
\begin{align}
&\cuV_{\mupart} := \{\mu \in\cuP(\R^{2})
\;\; \mu_{(t_0)} = \mu_0
\}\, ,\label{eq:cuV_mupart_def}\\
& \cuV_\nupart(\bR) 
:= \Big\{ \nu \in\cuP(\R^{3}) :
\E_\nu[(v,v_0)\ell'(v;v_0,w)] +\lambda (r_{11},r_{10})=\bzero,\;
\nu_{(w)} = \P_w,\;
\mu_{\star}(\nu)((-\infty,-\lambda)) = 0 
\Big\}.\label{eq:cuV_nupart_def}
\end{align}
\begin{theorem}
\label{thm:fin_dim_var_formula_k=1}
\label{thm:fin_dim_var_formula}
For $\delta>0$, define the event
\begin{equation}
    \Omega_\delta := 
    \big\{\bw \in\Ball_{\sfA_w\sqrt{n}}^n (\bzero) ,\; W_2(\hnu_\bw, \P_w) < \delta
    \big\}\, ,
\end{equation}
Let Assumptions \ref{ass:regime} to \ref{ass:theta_0} of Section~\ref{sec:assumptions} hold.
Then for any $\sfA_R,\sfa_L>0$, the following hold.
\begin{enumerate}
    \item (Minimax variational principle.)
    \label{item:thm_1_item_1}
With
   \begin{equation}
   \label{eq:Z_n_def}
       Z_n := Z_n(\cuA,\cuB;\sfA_R,\sfa_L) := |\cZ_n(\cuA,\cuB) \cap \cG_n(\sfA_R,\sfa_L)|,
   \end{equation}
  there exists a high-probability $\Omega_0'$ such that
    for any 
    $\cuA,\cuB$ open in the $W_2$ topology, 
  we have 
\begin{align}
 \lim_{\delta\to0}\lim_{n\to\infty}\frac{1}{n}\log\E\left[Z_n \one_{\Omega_\delta\cap \Omega_0'}\right]
 &\le
 -\hspace{-5mm}
\inf_{\substack{\mu \in \cuA\cap\cuV_{\mupart}\cap\cuG_\mupart 
} }
\inf_{
\nu \in\cuB \cap \cuV_\nupart(\bR(\mu))\cap \cuG_\nupart
}
\Phi_\star(\mu,\nu)
\end{align}
where 
\begin{equation}\label{eq:PhiStarThm}
\Phi_\star(\mu,\nu) :=    
\Phi_\mupart(\mu) +
 \sup_{s \in \cS_0(\nu)} \Phi_{\nupart}(\nu,\bR(\mu),s)
\end{equation}
with 
%
\begin{align}
\label{eq:S-Set-Def}
\cS_0(\nu) &:= \bigg\{ s \in\R :  s \in \left(0, L_{0}(\nu)^{-1}\right),\; \E_{\nu}\bigg[ \bigg(\frac{s \ell''(v; v_0,w)}{ 1 + s \ell''(v;v_0, w)}\bigg)^2\bigg] < \frac1\alpha\bigg\},\\
L_0(\nu) &:= -\Big( 0 \wedge \inf_{(v,v_0,w) \in \supp(\nu)} \ell''(v;v_0,w)\Big)
\\
   \cuG_\mupart  &:= \cuG_\mupart(\sfA_R)
   := 
\left\{\mu: \sfA_R^2 \succ \bR(\mu) \right\}\\
   \cuG_\nupart  
   &:= 
   \cuG_\nupart(\sfa_L)  :=
\left\{\nu : 
\E_{\nu}[\ell'(v,v_0,w)^2 ] > \sfa_L^2,
\right\}.
\label{eq:cuG_nupart_def}
\end{align}

\item (Finite dimensional variational principle.)
    \label{item:thm_1_item_2}
    Assume $\sfL_{\min} > 0.$
Fix open subsets $\cR \subseteq \sfS_{\succeq 0}^{2},$
$\cL \subseteq \R $, along with
a function $\psi :\R^{3}\to\R$ that is  locally Lipschitz and at most quadratic growth.
Let
\begin{align}
&\cuA_\cR := \cuA_\cR(\sfA_R):=\Big\{\mu \in\cuP(\R^2)  : \bR(\mu) \in\cR,
\sfA_R^2\succ \bR(\mu)
\Big\}\, ,\\
&\cuB_{\cL} := 
\cuB_{\cL}(\sfa_L)
:= \Big\{\nu \in \cuP(\R^3) : \E_{\nu}[\psi(v,v_0,w)] \in \cL,\;
    \E_\nu[\ell'(v,v_0,w)^2] > \sfa_L^2
    \Big\}\, .
\end{align}
Then, there exists a high-probability event $\Omega_0'$, such that 
\begin{equation}
\label{eq:computation_lb}
 \lim_{\delta\to0}\lim_{n\to\infty}\frac{1}{n}\log\E\left[|\cZ_n(\cuA_\cR,\cuB_\cL)|\one_{\Omega_\delta \cap\Omega_0'}\right]
 \le 
-\inf_{\substack{\bR \in\cR\\ r_{00}=\int t_0^2 \de\mu_0 }} \inf_{L \in\cL} 
\Phi_{\fin}(\bR,L)\, .
\end{equation}
where
\begin{align}
\nonumber
\Phi_{\fin}(\bR,L) :=
    \inf_{\substack{s\in(0,\infty)\\ g > \sfa_L^2}}
    \sup_{
    \substack{
    \gamma,\xi,\eta\in\R\\
    \bbeta \in \R^{2},\; \zeta \ge 0
    }
    }
    &
\bigg\{
\frac1{\alpha} \log(s)
+ \frac{(\xi-\zeta)}{\alpha}
-\lambda s
+ \gamma g
-\lambda  \bbeta^\sT(r_{11},r_{10})\\
    &
+ \eta\, L
    +  \frac1{2\alpha}\log\left(
    \frac{\alpha\, e\, g}{\schur(\bR,r_{00})}
    \right) 
    -\E_w[\log X(s,\bR; \gamma,\xi,\eta,\bbeta,\zeta)]
    \bigg\}\, ,\label{eq:PhiLb_First}
\end{align}
and
\begin{align}
  X(s,\bR; \gamma,\xi,\eta,\bbeta,\zeta) &:= 
    \int_{\{(v,v_0) : 1 + s \ell''(v,v_0,w) >0\}} \exp\big\{
 \gamma\ell'(v;v_0,w)^2
 + \bbeta^\sT(v,v_0)\ell'(v;v_0,w)
 + \eta\,\psi(v,v_0,w)\big\}\nonumber\\
 &\hspace{10mm}\cdot
  \exp\bigg\{\xi \frac{s \ell''}{1 + s\ell''}  
-
\zeta
\bigg(\frac{s \ell''}{1 + s\ell''}\bigg)^2
\bigg\}
    (1 + s \ell'')
    \gamma_{\bR} (\de v,\de v_0)\, .
\end{align}
\end{enumerate}
\end{theorem}

\begin{remark}
    The first point of Theorem \ref{thm:fin_dim_var_formula_k=1} 
    yields a rate function for which the dependence on $\mu,\nu$ is explicit (modulo the supremum over $s$). This is to be compared with the formula in \cite[Theorem 1]{asgari2025local} which contain a term proportional to the log
    potential $\int \log(x+\lambda) \mu_{\star}(\nu)(\de x)$.

    Theorem \ref{thm:fin_dim_var_formula_k=1}  is proven by deriving a variational
    principle for this log potential, whose optimization variable $s$ gives (at the optimum) the Stieltjis transform of $\mu_{\star}(\nu)$ at $-\lambda$.
\end{remark}

\subsection{Rate trivialization}

In this section, we will provide sufficient conditions   under which the rate function of the previous section `trivializes.'

Before doing so, let us introduce the definition of the optimal measure $(\mu^\opt,\nu^\opt)$ that we expect to minimize the rate function $\Phi_\star$ under rate trivialization.

For $v_0,w \in\R, s >0,$  and $\ell :\R^{3} \to\R$,
the proximal operator $\Prox_{\ell(\cdot, v_0, w)}(\,\cdot\,;s) :\R\to\R$ is defined by
\begin{equation}
\label{eq:ProxDefinition}
    \Prox_{\ell(\cdot, v_0, w)}(z;s):=\arg\min_{x\in\R}\left\{ \frac1{2s}(x-z)^2 + \ell(x,v_0,w)\right\}\, .
\end{equation}
In general,  $\Prox_{\ell(\cdot, \bv_0, w)}(\,\cdot\,;s)$ can be viewed as a set-valued function, but we will only apply it to cases in which the minimum is unique as we remark shortly.

\begin{definition}
\label{def:opt_FP_conds}
We say that a pair $(\mu^\opt, \nu^\opt) \in \cuP(\R^{3})\times \cuP(\R^{2})$ satisfy  the \emph{local optimality conditions} if for some $(s^\opt,\bR^\opt) \in (\R_{>0},\sfS_{\succ\bzero}^{2})$, $\bR^\opt=\Big(\begin{matrix}r_{11}^\opt& r_{10}^\opt\\
r_{01}^\opt & r_{00}\end{matrix}\Big)$, we have
   \begin{align}
   \label{eq:NuOptDef}
&\nu^\opt = \mathrm{Law}(v,g_0,w),\quad
v = \Prox_{\ell(\,\cdot\,;g_0,w)}( g; s^\opt),
\quad
\textrm{where}\quad\quad
(g,g_0, w) \sim \cN\left( \bzero_{2}, \begin{pmatrix}
   r_{11}^\opt , r_{10}^\opt\\
  r_{10}^\opt, r_{00}
\end{pmatrix} \right)\otimes \P_w,
\end{align}
and
\begin{equation}
   \mu^\opt_{(t_0)} = \mu_0 ,\quad\quad \mu^\opt_{\cdot|t_0} = \cN\big(r_{10}^\opt r_{00}^{-1} t_0, \schur(\bR^\opt,r_{00})\big) \quad\quad\textrm{for}\quad\quad t_0 \in\R,
\end{equation}
and the following equations are satisfied:
\begin{enumerate}
    \item (Stieltjes transform condition): 
\begin{align}
& \E_{\nu^\opt}\left[\frac{\ell''(v;g_0,w)}{1 + s^\opt \ell''(v;g_0,w)}\right]  + \lambda = \frac{1}{\alpha\, s^\opt},
\nonumber\\
& \E_{\nu^\opt}\left[\Big(\frac{s^\opt\ell''(v;g_0,w)}{1 + s^\opt \ell''(v;g_0,w)}\Big)^2\right] < \frac{1}{\alpha} \, ,
 \quad
s^\opt\in (0, \sfL_{\min}^{-1}).\label{eq:Stieltjis-Condition}
\end{align}
\item (Asymptotic stationarity):
\begin{align}
\label{eq:opt_fp_eqs}
    &\E_{\nu^\opt}\left[\ell'(v;g_0,w)(v,g_0)\right]  + \lambda (r_{11}^\opt,r_{10}^\opt)= \bzero
    \, .
\end{align}
\end{enumerate}
\end{definition}
Note that conditions \eqref{eq:Stieltjis-Condition} and 
\eqref{eq:opt_fp_eqs} have an immediate interpretation. 
Equation~\eqref{eq:Stieltjis-Condition} amounts to 
$\inf \supp(\mu_{\MP}(\nu^{\opt}))\ge -\lambda$ with
$s^\opt$ being the 
Stieltjes transform of $\mu_{\MP}(\nu^{\opt})$  evaluated at $-\lambda$
(recall the definition and interpretation of $\mu_{\MP}(\nu)$ in Section \ref{sec:definitions}).
Equation~\eqref{eq:Stieltjis-Condition} simply corresponds to the condition
$(\btheta,\btheta_0)^{\sT}\nabla\hR_n(\btheta)=0$ that holds at any stationary point.

\begin{remark}
The inequality condition in Eq.~\eqref{eq:Stieltjis-Condition} can also be rewritten
as
\begin{align}
\E\big\{\big(\Prox(g;g_0,w)-g\big)^2\big\} <\frac{1}{\alpha}\, ,
\end{align}
and is analogous of the so-called `replicon condition' in spin glass theory \cite{de1978stability}.
It appeared in the analysis of high-dimensional convex ERM in \cite[Appendix C]{BayatiMontanariLASSO}
where it was proven that it implies high-dimensional convergence of AMP under certain initializations,
as also  proven independently in \cite{bolthausen2014iterative}. 

Here it is related in a direct way to the requirement that the limiting spectral distribution
of the Hessian is supported on the positive semiaxis. For this reason, we will also refer to it as the 
`stability' condition.
\end{remark}

\begin{remark}
Despite $\ell$ being non-convex, the proximal operator in Eq.~\eqref{eq:NuOptDef}
is a well-defined function, in the sense that Eq.~\eqref{eq:ProxDefinition} has a unique minimizer.
Indeed, by Eq.~\eqref{eq:Stieltjis-Condition} we require $s^\opt < 1/\sfL_{\min}$,
so that $1 + s \ell''(v;v_0,w) > 0$ for all $v\in\R$, and therefore 
the optimization problem in Eq.~\eqref{eq:ProxDefinition} is strictly convex.
\end{remark}

The next theorem establishes that rate trivialization occurs for a variety of losses
provided $\alpha$ is larger than a sufficiently large constant. 
\begin{theorem}
\label{thm:trivialization_k=1}
Let Assumptions \ref{ass:regime} to \ref{ass:theta_0} of Section~\ref{sec:assumptions} hold, and 
recall $\Phi_\mupart, \Phi_{\nupart}$ defined in Eq.~\eqref{eq:phi_A_phi_B_defs} and $\Phi_\star$ defined 
in Eq.~\eqref{eq:PhiStarThm} of Theorem~\ref{thm:fin_dim_var_formula_k=1}.
Then:
\begin{enumerate}
    \item
    \label{item:thm_2_item_1}
    If $(\mu^\opt,\nu^\opt)$ satisfy the local optimality condition 
     of Definition~\ref{def:opt_FP_conds} with a corresponding pair $(s^\opt,\bR^\opt)$,
we have  $\mu^\opt \in\cuV_{\mupart}$ and $\nu^\opt \in\cuV_{\nupart}(\bR^\opt)$. Furthermore,
\begin{equation}
\label{eq:Phi_at_opt_is_0}
  \Phi_\star(\mu^\opt,\nu^\opt) = \Phi_\mupart(\mu^\opt) + \Phi_{\nupart}(\nu^\opt, \bR^\opt, s^\opt) = 0.
\end{equation}

\item (Local trivialization at large $\alpha$)
    \label{item:thm_2_item_2}
Define
\begin{equation}
\label{eq:alpha_0_hat_s_def}
\alpha_0 :=  \left(\frac{\sfL_{\min}}{\sfL_{\min} \vee \sfL_{\max}}\right)^2,\quad\quad \textrm{and}\quad\quad
    \tau_0 := \frac12 \left( \frac{1}{\sfL_{\min} \vee \sfL_{\max}}\right),
\end{equation}
and 
    \begin{equation}
        \cuT := \big\{(\mu,\nu) : \schur(\bR(\mu), r_{00}) \le \tau_0^2 \E_{\nu}[\ell'(v,v_0,w)^2]\big\}.
    \end{equation}

If $\alpha > \alpha_0$, then $\Phi_\star$ trivializes over $\cuT$: i.e., 
for any $(\mu,\nu) \in \cuT$ with $\mu\in\cuV_\mupart$, $\nu\in\cuV_\nupart(\bR(\mu))$,
    \begin{equation}
    \Phi_\star(\mu,\nu) \ge 0,
    \end{equation}
    with equality only if $(\mu,\nu) = (\mu^\opt,\nu^\opt)$, where $(\mu^\opt,\nu^\opt)\in\cuT$ solves the equations of Definition~\ref{def:opt_FP_conds}.

\item (Global trivialization at large $\alpha$)
    \label{item:thm_2_item_3}
Let 
\begin{equation}
    \mathrm{B}_\star(\nu,\bR) := \tau_0 \big(\E_\nu[\ell''] + \lambda\big)  + \frac12 \log\left(\frac{\schur(\bR,r_{00})}{\tau_0^2 \E_\nu[\ell'^2]}\right)
\end{equation}
and
\begin{equation}
    \mathrm{G}_\star(\bR) := 
   \frac{ \big\|\E_{(u,u_0)\sim\cN(\bzero,\bR)}\big  [ (u,u_0)\ell'(u;u_0,w) \big]  
     + \lambda (r_{11},r_{10})
    \big\|_2^2}
    {
    2\,\|\bR\|_\op
    \left(\tau_0^{-1}\, \schur(\bR,r_{00})^{1/2} + \|\ell'\|_{\Lip} \Tr(\bR)^{1/2}\right)^2
    }.
\end{equation}

If $(\mu,\nu)\in\cuT^c$, then
\begin{equation}
    \Phi_\star(\mu,\nu) 
    \ge 
    \Phi_\mupart(\mu) + 
    \rmG_\star(\bR(\mu))  - \frac1\alpha \rmB_{\star}(\nu,\bR(\mu)) - \frac1\alpha\log(\alpha)
\end{equation}
\end{enumerate}
\end{theorem}

\begin{remark}
We note that $\Phi_\mupart(\mu)$ and $\rmG_\star(\bR(\mu))$ are positive quantities and therefore the last 
result implies rate trivialization for $\alpha$ large enough, provided we can upper bound 
$\mathrm{B}_\star(\nu,\bR)$ and lower bound   $\rmG_\star(\bR(\mu))$ uniformly over $\mu,\nu\in\cuT^c$.
In some cases, by an a priori estimate, it is sufficient  to prove such bounds over subsets of $\cuT^c$.
\end{remark}

\begin{remark}
The fact that the optimality conditions of Definition \ref{def:opt_FP_conds} (in particular under the
replicon/stability condition) characterize the empirical risk minimizer
was proven for problems of the class considered here in \cite{vilucchio2025asymptotics}, under some technical conditions
on the solution itself. Theorem \ref{thm:trivialization_k=1} recover this result (at $\alpha$ large enough)
while also establishing that no other local minima exist. In particular, our result can be used to
establish convergence of gradient descent algorithms to the global minimum.
\end{remark}

\section{Non-convex M-estimation}
\subsection{General theory}

To give a more concrete version of the results above, we specialize to 
the case of learning a linear model by minimizing a loss that depends uniquely on
the prediction error $\btheta^{\sT}\bx_i-y_i$.
Namely, we assume $y_i = \btheta_0^{\sT}\bx_i+w_i$ and, for  
$\ell:\R\to\R$, we form the empirical risk
\begin{equation}
\label{eq:ERM_translation_invariant}
    \hat R_n(\btheta) := \frac1n\sum_{i=1}^n\ell\big(\bx_i^\sT\btheta- y_i\big)= \frac1n\sum_{i=1}^n\ell\big(\bx_i^\sT(\btheta-\btheta_0)- w_i\big)\, .
\end{equation}
The `M-estimator' $\hbtheta=\arg\min\hR_n(\btheta)$
 has been the object of a long line of work in high-dimensional statistics, see e.g. \cite{huber1973robust,mammen1989asymptotics,donoho2016high,el2013robust}.

For a given $\sfA_R,\sfa_L>0$,
we redefine the set $\cG_n$ of Eq.~\eqref{eq:well_behaved} in this setting
\begin{align}
    \cG_n(\sfA_R,\sfa_L) :=
    &\left\{
   \btheta \in\Ball^d(\btheta_0, \sfA_R),\;
   \frac1n\sum_{i=1}^n \ell'(\bx_i^\sT(\btheta-\btheta_0) - w_i)^2  > \sfa_L^2,\;  
   \sigma_{\min}\big(\bJ \bg_0(\bX(\btheta-\btheta_0), \bw)\big) > e^{-\overline{o}(n)}
    \right\}\, ,\nonumber
\end{align}
where $\bg_0(\bv, \bw) :=\bv^\sT\ell'(\bv-\bw)/n$.

The next result is proven by applying Theorem~\ref{thm:trivialization_k=1}
to the M-estimation setting.
\begin{theorem}
\label{cor:robust_regression}
Consider the ERM problem in~\eqref{eq:ERM_translation_invariant}
under Assumptions~\ref{ass:regime}-\ref{ass:theta_0}
(in particular, $\|\btheta_0\|_2^2\to r_{00}$)
and introduce the notation
$\sfL_{\vee} := \sfL_{\max}\vee\sfL_{\min}$.
Fix parameters $\sfA_R,\sfa_L >0$.
If there exists $\alpha_\star >1$ such that
\begin{equation}
    \alpha_\star > 
    \frac{ 
    9 \sfL_\vee^2
    }{\E[\ell''(\beta Z - w)]^2}
    \log\left( \frac{4\, e\,\beta^2 \sfL_\vee^2 \alpha_\star^2}{ \sfa_L^2}\right) 
    \quad \textrm{for all} \quad \beta \in \left[ \frac{\sfa_L}{2 \sfL_\vee} ,\sfA_R\right]
    \quad
\textrm{where}\quad
(Z,w)\sim\cN(0,1)\otimes \P_w,
\end{equation}
then
for all $\alpha > \alpha_\star$,
the following hold:
\begin{enumerate}
    \item 
    \label{item:cor_1_item_1}
For any $\cuA,\cuB\in\cuP(\R)\times \cuP(\R^{2})$ open in the $W_2$ topology,
we have for some $c_0>0$ and some high probability set $\Omega_1'$ 
\begin{equation}
\lim_{\delta\to0}\lim_{n\to\infty}\frac{1}{n}\log\E\left[Z_n(\cuA,\cuB;\sfA_R,\sfa_L) \one_{\Omega_\delta\cap\Omega_1'}\right] \le 0,
\end{equation}
where the inequality is strict 
unless there exist
$(\mu^\opt,\nu^\opt)\in \overline{\cuA\times\cuB}$  (the closure of $\cuA\times\cuB$)
given by
\begin{align}
\nu^\opt &= \mathrm{Law}(v,g_0,w),\quad
v = g_0+w+v_1\, ,\;\;\;\; v_1=\Prox_{\ell}( g_1; s^\opt),
\label{eq:NuoptRegr}\\
&\quad\quad\quad
(g_1,g_0, w) \sim \cN(0,(\rho^{\opt})^2)\otimes \cN(0,r_{00})\otimes \P_w,\nonumber\\
 \mu^\opt_{(t_0)} &= \mu_0 ,\quad\quad \mu^\opt_{\cdot|t_0} = \cN\big(t_0, (\rho^{\opt})^2\big) \quad\quad\textrm{for}\quad\quad t_0 \in\R,
\end{align}
with $(\rho^{\opt},s^\opt)$ satisfying the following  equations
\begin{align}
\label{eq:ST_robust}
& \E_{\nu^\opt}\left[\frac{\ell''(v_1)}{1 + s^\opt \ell''(v_1)}\right]   = \frac{1}{\alpha\, s^\opt},
\quad 
 \E_{\nu^\opt}\left[\Big(\frac{s^\opt\ell''(v_1)}{1 + s^\opt \ell''(v_1)}\Big)^2\right] < \frac{1}{\alpha},
 \quad
s^\opt\in (0, \sfL_{\min}^{-1})\, ,\\
    &\E_{\nu^\opt}\left[\ell'(v_1)v_1\right]  = 0
    \, .\label{eq:ST_robust_2}
\end{align}
As a consequence, the only possible limit points of $(\hat\mu(\hat\btheta),\hnu(\hat\btheta))$ for local minima $\hat\btheta \in\cG_n(\sfA_R,\sfa_L)$ are measures solving the system of equations 
 \eqref{eq:ST_robust} to \eqref{eq:ST_robust_2}.
\item
    \label{item:cor_1_item_2}
Further, if Eqs.~\eqref{eq:ST_robust}, \eqref{eq:ST_robust_2}  have a unique solution $(\mu^\opt,\nu^\opt)$,
then for all $\eps>0$,
\begin{equation}
    \lim_{n,d\to\infty}\P\big(\exists \;\textrm{\rm{local min.}}\; \hat\btheta\in\cG_n(\sfA_R,\sfa_L) :  W_{2}(\hmu(\hat\btheta), \mu^\opt) + W_2(\hnu(\hat\btheta), \nu^\opt) > \eps\big) = 0.
\end{equation}
\end{enumerate}
\end{theorem}

\subsection{Tukey regression}
\label{sec:tukey}

As a concrete illustration of our results, we consider robust regression with Tukey loss. 
Namely, given data $(\bx_i,y_i)$, $i\le n$, we  minimize the empirical risk \eqref{eq:ERM_translation_invariant}
where $\ell(\,\cdot\,)=\ell_\tuk(\,\cdot\,;\kappa)$ is the Tukey loss defined by
\begin{equation}
    \ell_\tuk(t; \kappa) =  \begin{cases}
    \displaystyle
       \frac{\kappa^2}{6} - \frac{\kappa^2}{6} \left( 1 - \frac{t^2}{\kappa^2}\right)^3 & \mbox{ if }|t| \le \kappa,\\
    \displaystyle\frac{\kappa^2}{6}        & \mbox{ if } |t| > \kappa,
    \end{cases}
\end{equation}
for  $\kappa > 0$ a parameter.
Notice that the maximum value of the loss is $\kappa^2/6.$

In what follows, we let $(\mu^\opt,\nu^\opt, r^\opt,s^\opt)$ denote the solution (if it exists uniquely) of the fixed point equations in Theorem~\ref{cor:robust_regression}.
\begin{corollary}[Tukey regression]
\label{prop:tukey}
Fix  $\kappa,\sfA_R>0$.
Assume $\btheta_0$ satisfies Assumption~\ref{ass:theta_0}, and that $\P_w$ has a finite second moment and satisfies,
for some $b_0 \in (0,\kappa/2),b_1 >0$,
\begin{equation}
   \P_w\left( |w| \in (b_0 ,\kappa - b_0) \right) > b_1.
\end{equation}
For open subsets $\cE \subseteq  (0,\kappa^2/6]$, and $\cR \subseteq (0,\sfA_R],$
let $\cZ_{\cE,\cR}$ be the set of minimizers achieving an estimation error in $\cR$ and a training error in $\cE$.
Namely:
   \begin{equation}
       \cZ_{\cE,\cR} := \left\{
      \btheta : \grad \hat R_n(\btheta) = \bzero,\; 
      \grad^2 \hat R_n(\btheta) \succeq \bzero,\; 
        \hat R_n(\btheta) \in\cE ,\;
        \|\btheta - \btheta_0\|_2 \in \cR
       \right\}.
   \end{equation} 
   \begin{enumerate}
       \item 
       With the definitions of Theorem~\ref{thm:fin_dim_var_formula_k=1}, there exists $\alpha_1 \equiv \alpha_1(\sfA_R,\kappa)$ 
       and a high probability event $\Omega_2'$ 
       such that for all $\alpha > \alpha_{1}$, $\cE \subseteq  (0,\kappa^2/6]$, and $\cR \subseteq (0,\sfA_R],$
\begin{equation}
 \lim_{\delta\to0}\lim_{\substack{n\to\infty \\ n/d \to\alpha}}\frac{1}{n}\log\E\left[|\cZ_{\cE,\cR}| \one_{\Omega_{\delta}\cap\Omega_2'}\right]
 \le
-\inf_{\substack{ \rho  \in \cR \\
\iota \in \cE
}}  
\Phi_{\tuk}(\rho;\iota)\, .\label{eq:EZ_Tukey}
\end{equation}
where 
\begin{align}
\Phi_{\tuk}(\rho;\iota) :=
    \inf_{\substack{s\in(0,\infty)\\ g \in [0,\frac{16\,\kappa} {25 \sqrt{5}} ]}}
    \sup_{
    \substack{
    \eta, \gamma, \xi,\beta  \in\R\\
    \zeta\ge 0
    }
    }
\bigg\{
&
 \gamma g^2
+ \frac{\xi - \zeta}{\alpha}
- \eta \,\iota
    +  \frac1{2\alpha}\log \left(
    \frac{
    \alpha\, e\,s^2\, t^2}{\rho^2}\right)
    -\E_w[\log X(s,\rho; \gamma,\xi,\eta,\beta,\zeta)]
    \bigg\}\, ,\label{eq:PhiLb_First}
\end{align}
with 
\begin{align}
  X(s,\rho; \gamma,\xi,\eta,\beta,\zeta) &:= 
    \int_{\{v : 1 + s \ell''(v-w) >0\}} \exp\big\{
 \gamma\ell'(v-w)^2
 + \beta v \ell'(v- w)
 + \eta\,\psi\big\}\\
 &\hspace{10mm}\cdot
  \exp\bigg\{\xi \frac{s \ell''}{1 + s\ell''}  
-
\zeta
\bigg(\frac{s \ell''}{1 + s\ell''}\bigg)^2
\bigg\}
    (1 + s \ell'')
    \gamma_{\rho} (\de v)\, .
\end{align}
 
In particular, if there exist unique $(\rho^\star,\iota^\star) \in(0,\sfA_R]\times (0,\kappa^2/6]$ such that
\begin{align}
  &\Phi_{\tuk}(\rho,\iota) \ge 0   \quad\quad \forall(\rho,\iota)  \in (0,\sfA_R]\times (0,\kappa^2/6]=:\cR\times\cE
\end{align} 
with equality if and only if $(\rho,\iota) = (\rho^\star,\iota^\star)$, then for any $\eps >0$:
       \begin{equation}
       \lim_{n\to\infty}\P\Big( \exists \hat\btheta \in\cZ_{\cE,\cR} \;\textrm{s.t.}\; \big| \|\hat\btheta-\btheta_0\|_2 - \rho^{\star}\big| + 
       \big| 
       \hat R_n(\hat\btheta)- \iota^\star \big| > \eps
       \Big) = 0.
       \end{equation}
       \item 
       Assume that for all $\beta \in (0,\sfA_R]$,
       \begin{equation}
           \E[\ell''_\tuk(\beta Z + W ; \kappa)] >0 ,\quad\quad (Z,W) \sim \cN(0,1)\otimes \P_w.
       \end{equation}
       Then there exists some $\alpha_2 \equiv \alpha_2(\sfA_R,\kappa)>1$ so that 
       for all $\alpha >\alpha_2$, 
       with $\cuM^\opt \times\cuV^\opt$ denoting the collection of 
       $(\mu^\opt,\nu^\opt)$ satisfying Definition~\ref{def:opt_FP_conds}, 
       we have for any $\eps>0$, 
       $\cE \subseteq  (0,\kappa^2/6]$, and $\cR \subseteq (0,\sfA_R],$
       \begin{equation}
       \lim_{n\to\infty}\P\Big( \exists \hat\btheta \in\cZ_{\cE,\cR} \;\textrm{s.t.}\; W_2(\hnu(\hat\btheta), \cuV^\opt) + 
       W_2(\hmu(\hat\btheta), \cuM^\opt) > \eps
       \Big) = 0.
       \end{equation}
   \end{enumerate}
\end{corollary}

In the following section, we compare the predictions of the 
last corollary with numerical experiments. In particular, we
will consider local minima with a value of training error below
a certain preassigned level $\iota_\star$, i.e. use $\cE=(-\infty,\iota_\star]$
in Eq.~\eqref{eq:EZ_Tukey}. 

    \begin{figure}[t] 
        \centering 
        \begin{subfigure}[t]{0.49\textwidth} 
            \centering
            \includegraphics[width=\textwidth]{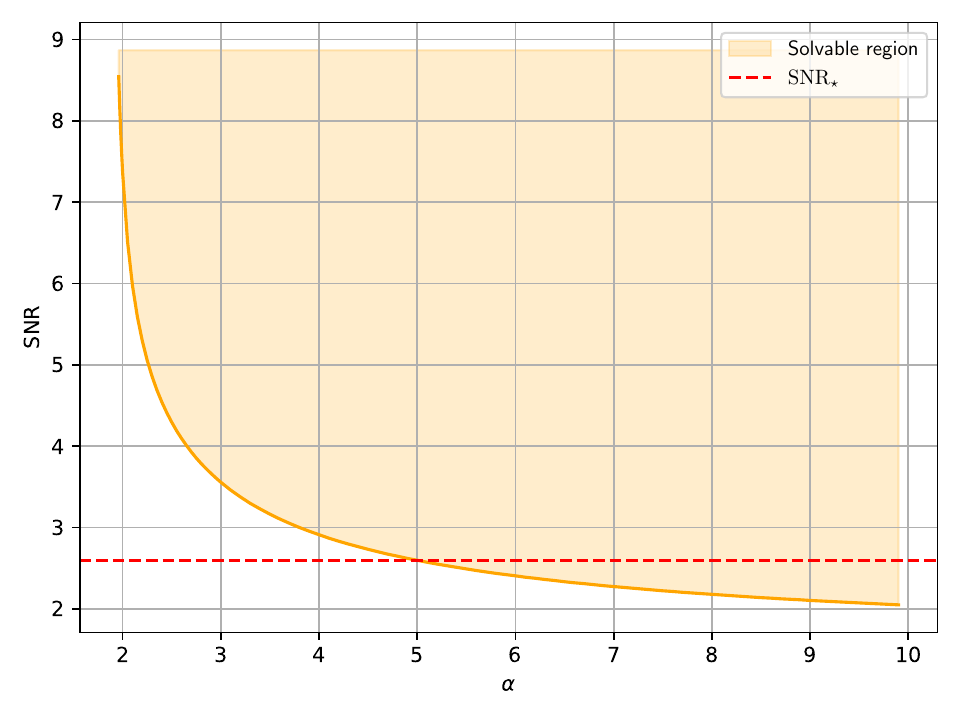} 
            \caption{
            Solving the optimality conditions of Definition~\ref{def:opt_FP_conds} for Tukey regression.
            The shaded region corresponds to values of $(\alpha,\SNR)$ for which the iterative scheme of Section \ref{sec:Numerical} converges to a solution $(\rho^\star,s^\star)$.
            The dashed line corresponds to the value $\SNR=\SNR_\star$ at which we produce all remaining plots of this section. 
            }
            \label{fig:FP_region}
        \end{subfigure}%
        \hfill 
        \begin{subfigure}[t]{0.49\textwidth} 
            \centering
            \includegraphics[width=\textwidth]{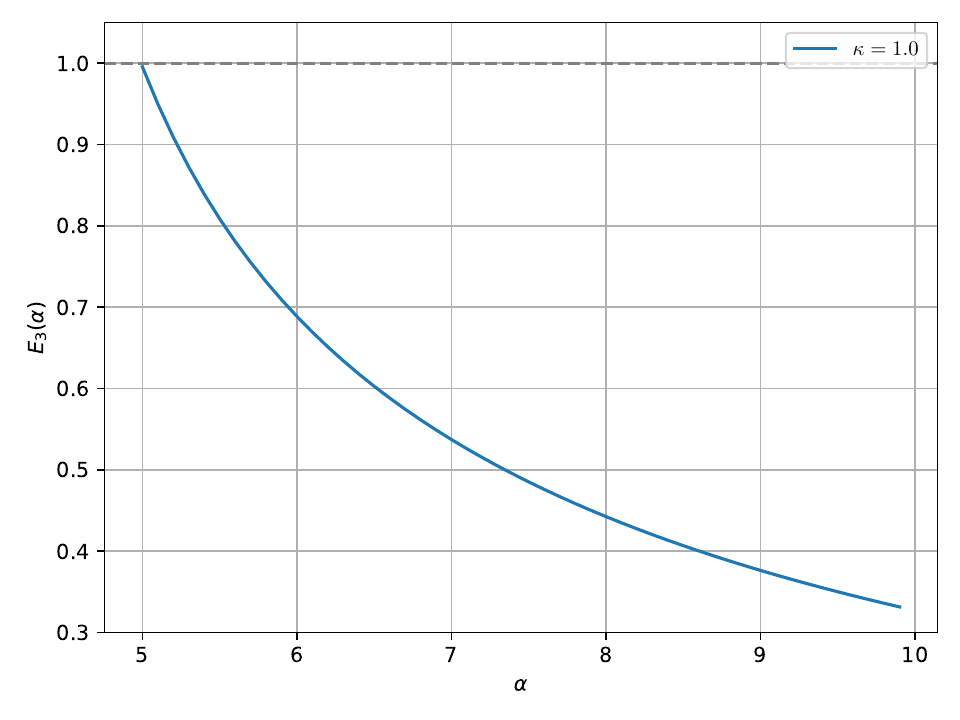} 
            \caption{
The stability criterion $E_{3}(\alpha)$ of Eq.~\eqref{eq:E3} as a function of $\alpha$ for the solution $(\nu^\star,s^\star)$ found by our numerical scheme. The failure of the equations to be solvable past the threshold is a result of the stability constraint $E_3(\alpha) \le 1$ being violated.
            }
            \label{fig:FP_constraint}
        \end{subfigure}
        \caption{Solving the optimality conditions of Definition~\ref{def:opt_FP_conds} for the Tukey loss. }
        \label{fig:FP_eqs}
    \end{figure}
%
%
\section{Numerical analysis}
\label{sec:Numerical}

In this section present results of numerical experiments carried out within
the Tukey regression setting of Section \ref{sec:tukey}.
Throughout, we fix $\btheta_0 \in\S^{d-1}$ (the unit sphere in dimension $d$), whence $r_{00}=1$ in Assumption~\ref{ass:theta_0}. The choice of $\btheta_0$ is irrelevant by rotation invariance and hence we can set,
without loss of generality, $\mu_{0}= \normal(0,1)$ (corresponding to $\btheta_0\sim\Unif(\S^{d-1})$). 
Whenever we carry out experiments that require running multiple runs, $\btheta_0$ is kept fixed across runs. 

We take the distribution $\P_w$ to be discrete as follows:
for some $A>0,$
\begin{equation*}
 w \in \{- 10 A, - A, +A, + 10A\} \quad\quad \textrm{with prob.}\quad\quad\{0.01, 0.49,0.49,0.01\}\quad\quad\textrm{respectively}.
\end{equation*}
The value $A$ will vary in our plots and will often be specified implicitly through  $\SNR:=\|\btheta_0\|_2^2/\Var(w)\propto  A^{-2}$.
Throughout we will take $\kappa =1$, which yields $\sfL_{\min}^{-1} = 1.25$.

    \begin{figure}[t] 
        \centering 
        \begin{subfigure}{0.49\textwidth} 
            \centering
            \includegraphics[width=\textwidth]{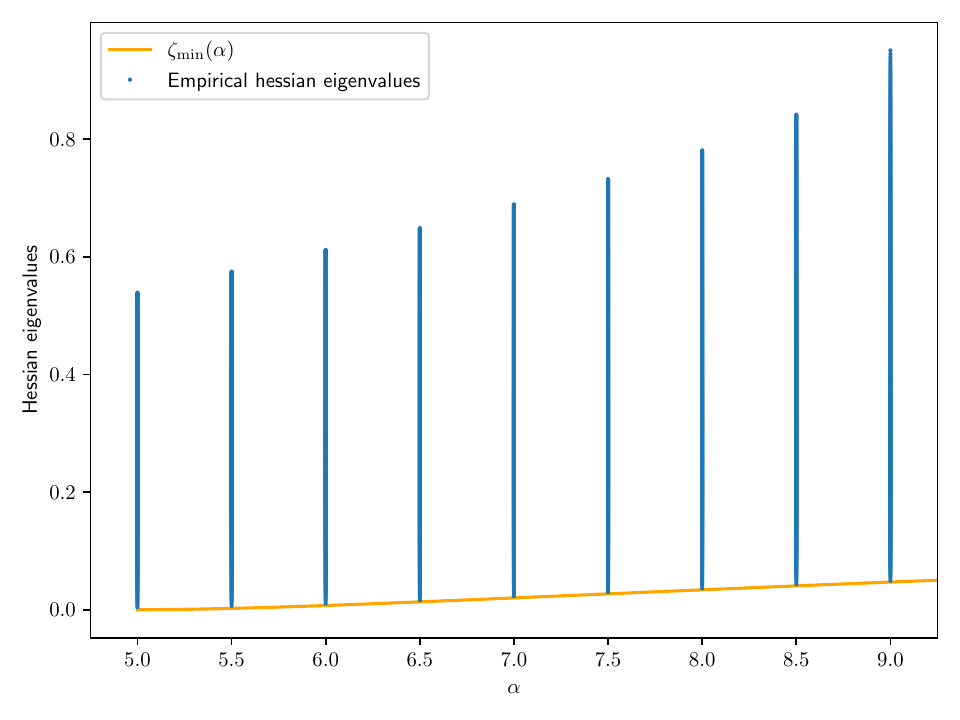} 
        \end{subfigure}%
        \hfill 
        \begin{subfigure}{0.49\textwidth} 
            \centering
            \includegraphics[width=\textwidth]{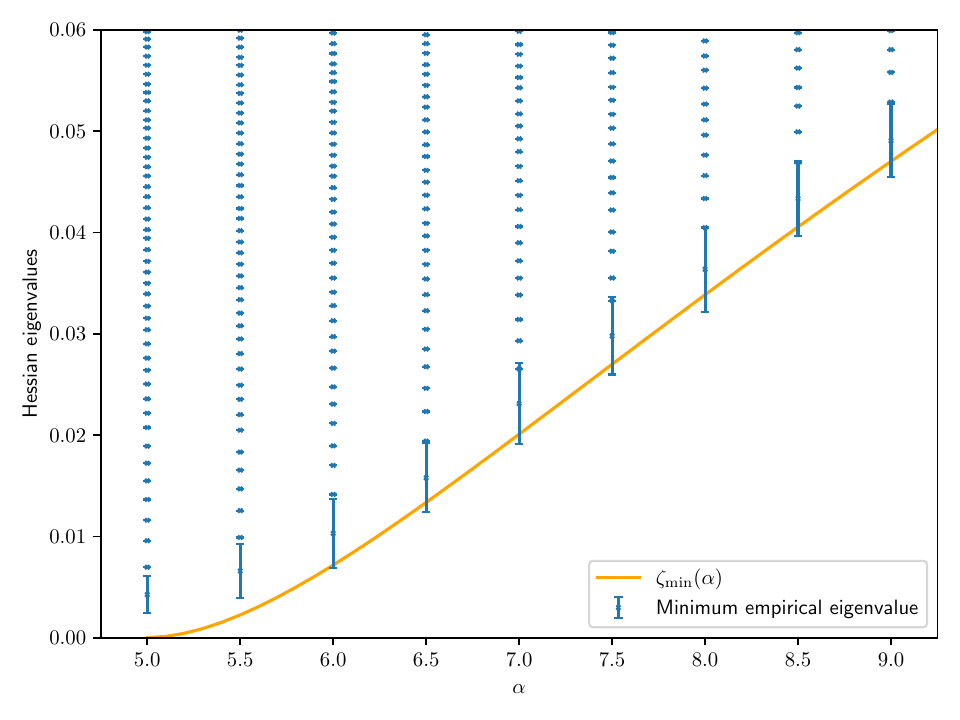} 
        \end{subfigure}
        \caption{Spectrum of the Hessian of the empirical risk $\nabla^2 \hR_n(\btheta)$ 
        for Tukey regression, compareed with theoretical prediction.
        The Hessian eigenvalues $\lambda^{(j)}_1\le \lambda^{(j)}_2\le 
        \cdots\le \lambda^{(j)}_d$ 
        are computed using $n=10000$, for each of $50$ trials:
        $j\in\{1,\dots,50\}$. For each $i\le d$, we plot the average and
         standard deviation of
        the $i$-th eigenvalue,  $\lambda^{(j)}_i$, over the $50$ trials.     
        Both plots are produced with the same data, with the right plot 
        zoomed to the left-edge of the support.  Here  $\SNR = \SNR_\star$ and
         $\kappa = 1.0.$  The continuous line reports the theoretical prediction
         for the lower edge of the asymptotic spectral distribution. 
        }
        \label{fig:FP_hessian}
    \end{figure}
 
\paragraph{Analysis of the fixed-point equations (Figures~\ref{fig:FP_eqs} and~\ref{fig:FP_hessian}).}
We begin by numerically solving the optimality conditions for the  Tukey regression problem. These are given by Eqs.~ \eqref{eq:ST_robust} to \eqref{eq:ST_robust_2} 
in Theorem~\ref{cor:robust_regression}, with $\ell(\,\cdot\,)=
\ell_{\tuk}(\,\cdot\,;\kappa)$. Throughout, we use the parameterization $\rho^2 := r - r_{00}$.

We search for a numerical solution by the following iterative scheme. We initialize to some value $\rho_0>0$, $s_0\in [0,\sfL_{\min}^{-1}),$
and iterate the following equations until convergence
\begin{align}
\rho_{t+1} &=  \left(\alpha s_t^2 \E[\ell'(v_t -w)^2] \right)^{1/2},\quad\quad
s_{t+1} = \left(\alpha \E\left[\frac{\ell''(v_{t} - w)}{1+s_t \ell''(v_{t} -w)}\right] \right)^{-1} \wedge \sfL_{\min}^{-1},\nonumber\\
v_t&:= \Prox_{\ell(\cdot - w)}(\rho_t G;s_t)\, ,\;\;\;  (G,w)\sim\normal(0,1)\otimes \P_w\, ,\label{eq:NumScheme}
\end{align}
for $t\in \{0,1,\dots,T_{\max}\}$ with $T_{\max} = 1000.$
It is not hard to show that the fixed points of this 
iteration, if any, are solutions of the equality conditions in Eqs.~\eqref{eq:ST_robust} and \eqref{eq:ST_robust_2} (see
Appendix~\ref{sec:numerical_details}).
 Finally, we implement the iteration of Eqs.~\eqref{eq:NumScheme} by
 approximating the expectations via numerical integration. 

If the  the iteration converges (within a preassigned tolerance) to a fixed point  $(\rho^\star,s^\star)$ 
we check the inequality constraint
\begin{equation}
\label{eq:E3}
    E_3(\alpha):=\alpha\E\left[\left(\frac{s^\star \ell''(v^\star-w)}{1 + s^\star \ell''(v^\star - w)}\right)^2\right] \le 1,\
\;\;\; v^\star:= \Prox_{\ell(\cdot - w)}(\rho^\star G;s^{\star})\, .
\end{equation}
If this condition is verified, we declare that the solution is \emph{stable.}

    \begin{figure}[t] 
        \centering 
        \begin{subfigure}[t]{0.32\textwidth} 
            \centering
            \includegraphics[width=\textwidth]{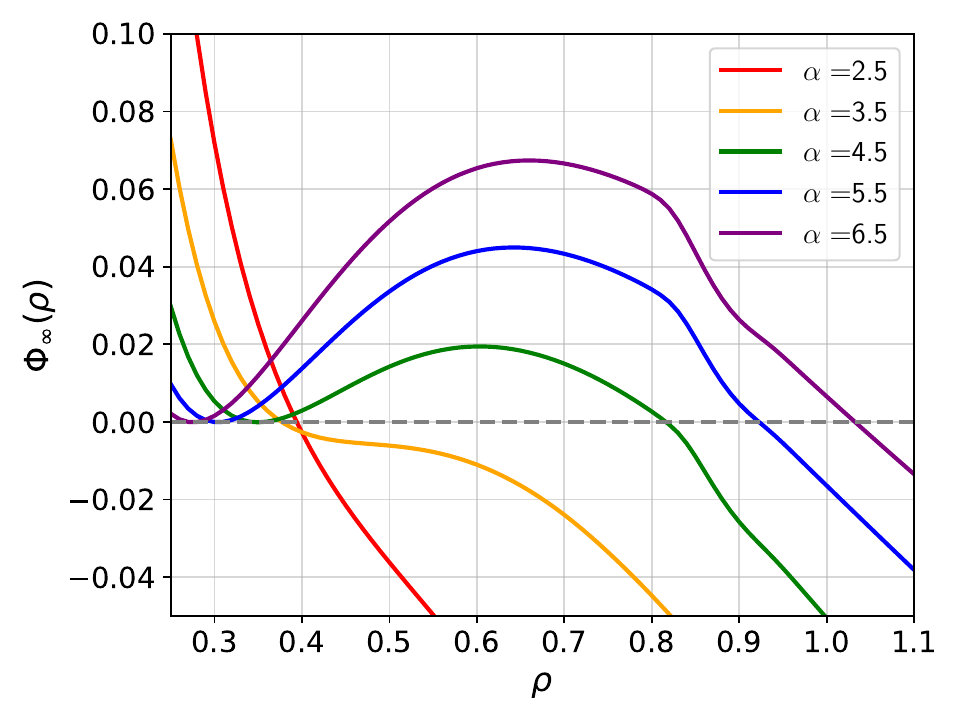} %
            \caption{}
            \label{fig:rate_func_1}
        \end{subfigure}%
        \hfill 
        \begin{subfigure}[t]{0.32\textwidth} 
            \centering
            \includegraphics[width=\textwidth]{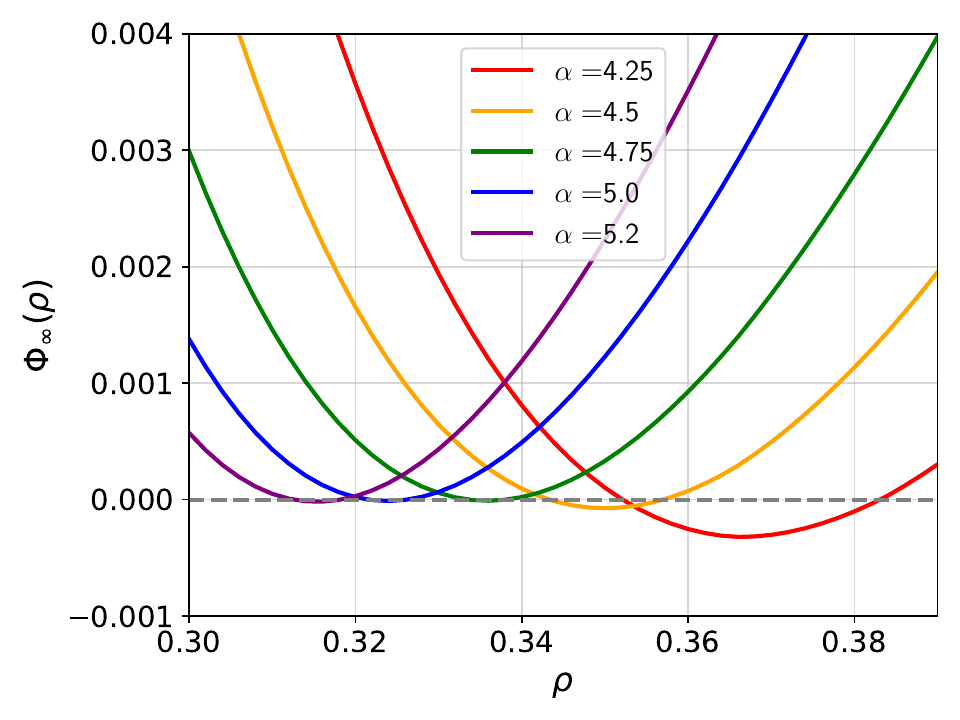} %
            \caption{}
            \label{fig:rate_func_2}
        \end{subfigure}
        \begin{subfigure}[t]{0.32\textwidth} 
            \centering
            \includegraphics[width=\textwidth]{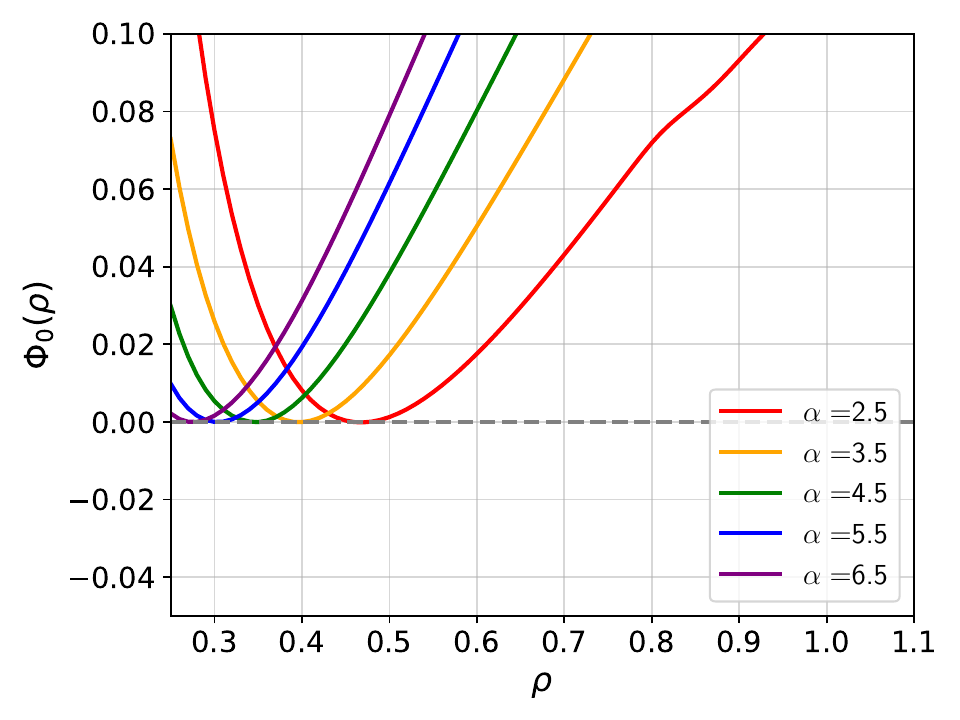} %
            \caption{}
            \label{fig:rate_func_3}
        \end{subfigure}
        
        \caption{
       Negative exponential growth rate of the number of local minimizers of
        the empirical risk $\hR_n(\btheta)$ for Tukey regression. 
       Left: $\Phi_{\infty}(\rho)$ of Eq.~\eqref{eq:phi_infty_def} vs. $\rho$ for various values of $\alpha$. 
       Center: zoom for values of $\alpha$  close to $\alpha =5$. 
       Right:  $\Phi_{0}(\rho)$ of Eq.~\eqref{eq:phi_0_def} vs. $\rho$ for several values 
       of $\alpha$. 
        }
            \label{fig:rate_function_plots}
    \end{figure}

Figure~\ref{fig:FP_eqs} shows the results of this numerical solution of the optimality equations.
In Figure~\ref{fig:FP_region}, we plot the region of the $(\alpha,\SNR)$-plane in which 
we find a stable solution of the equations using the iterative scheme. For any fixed $\SNR$,
a solution exists when $\alpha$ becomes larger than a threshold $\alpha_{\str}(\SNR)$.

Figure~\ref{fig:FP_constraint} traces $E_3(\alpha)$ of Eq.~\eqref{eq:E3} at the obtained solution $(s^\star,\rho^\star)$ as a function of $\alpha$.
Here, we choose $\SNR = \SNR_\star\approx 2.73$ corresponding to the dashed line in Figure~\ref{fig:FP_region}. 
This value of $\SNR$ is chosen so that we obtain $\alpha_{\str}(\SNR_{\star}) \approx 5.$
From this plot, we observe that $E_3(\alpha)<1$ for $\alpha>\alpha_{\str}(\SNR_{\star})$ and
$E_{3}(\alpha)\uparrow 1$  as $\alpha\downarrow\alpha_{\str}(\SNR_{\star})$ indicating that the 
threshold $\alpha_{\str}$ is determined as the point at which the inequality condition fails.

Note if $(s^\star,\rho^{\star})$ is a fixed point of Eq.~\eqref{eq:NumScheme} 
(hence it satisfies Eq.~\eqref{eq:ST_robust}), 
then $s_{\star}$ is the Stieltjes transform at $0$ of $\mu_{\star}(\nu^\star)$ (with $\nu_{\star}$ given by 
Eq.~\eqref{eq:NuoptRegr}).
Furthermore, the condition  $E_3(\alpha) \le 1$ (the inequality in Eq.~\eqref{eq:ST_robust}) 
implies that $\inf\supp(\mu_\star(\nu^\opt)) \ge 0$.

    \begin{figure}[t] 
        \centering 
        \begin{subfigure}{0.48\textwidth} 
            \centering
            \includegraphics[width=\textwidth]{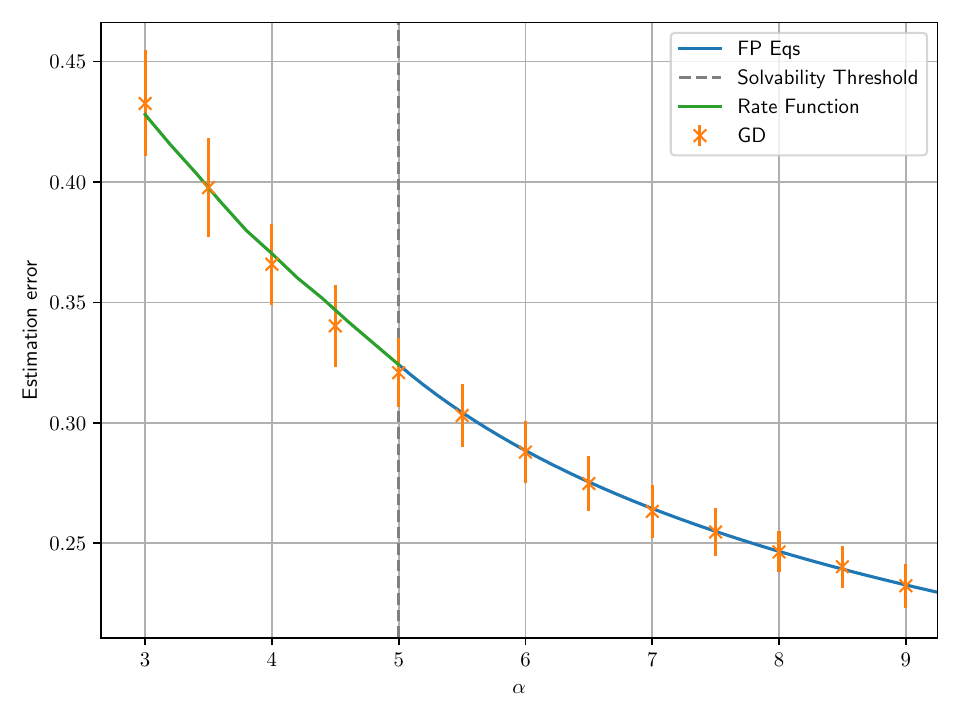} 
        \end{subfigure}%
        \hfill 
        \begin{subfigure}{0.48\textwidth} 
            \centering
            \includegraphics[width=\textwidth]{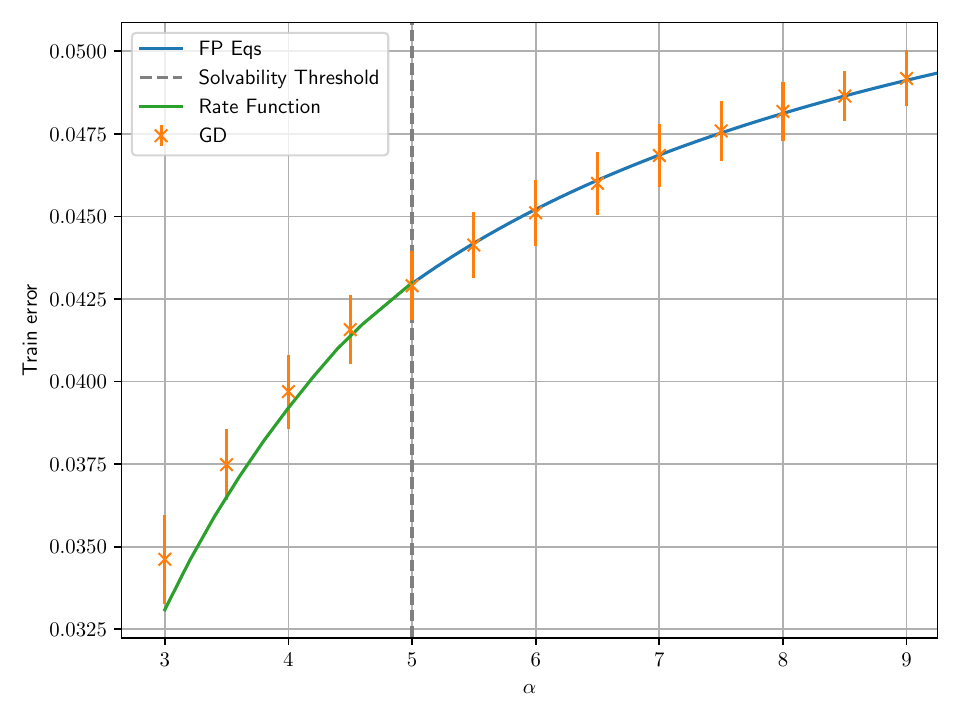} 
        \end{subfigure}
        \caption{Predictions for the estimation error (left) and train loss (right) at the global minimum, compared
        with empirical results with GD, for the same quantities,
        namely $\|\hat\btheta_{\GD}-\btheta_0\|_2$ and $\hat R_n(\hat\btheta_\GD)$. 
        For GD, we used  $d=500$ and an average of $100$ trials, with error 
        bars indicating the standard deviation.
       For $\alpha > \alpha_{\str}(\SNR_\star) = 5,$ the predictions are
        produced by solving the  equations (\textsf{FP Eqs}). The dashed lines
        correspond to the stability  threshold. 
        Theoretical predictions are proven to be asymptotically exact for
        $\alpha$ sufficiently large, but are only an heuristic approximation
        for $\alpha< \alpha_{\str}(\SNR_\star)$.}
        \label{fig:preds_full}
    \end{figure}

    \begin{figure}[t] 
        \centering 
        \begin{subfigure}[t]{0.245\textwidth} 
            \centering
            \includegraphics[width=\textwidth,height=1.18\textwidth]{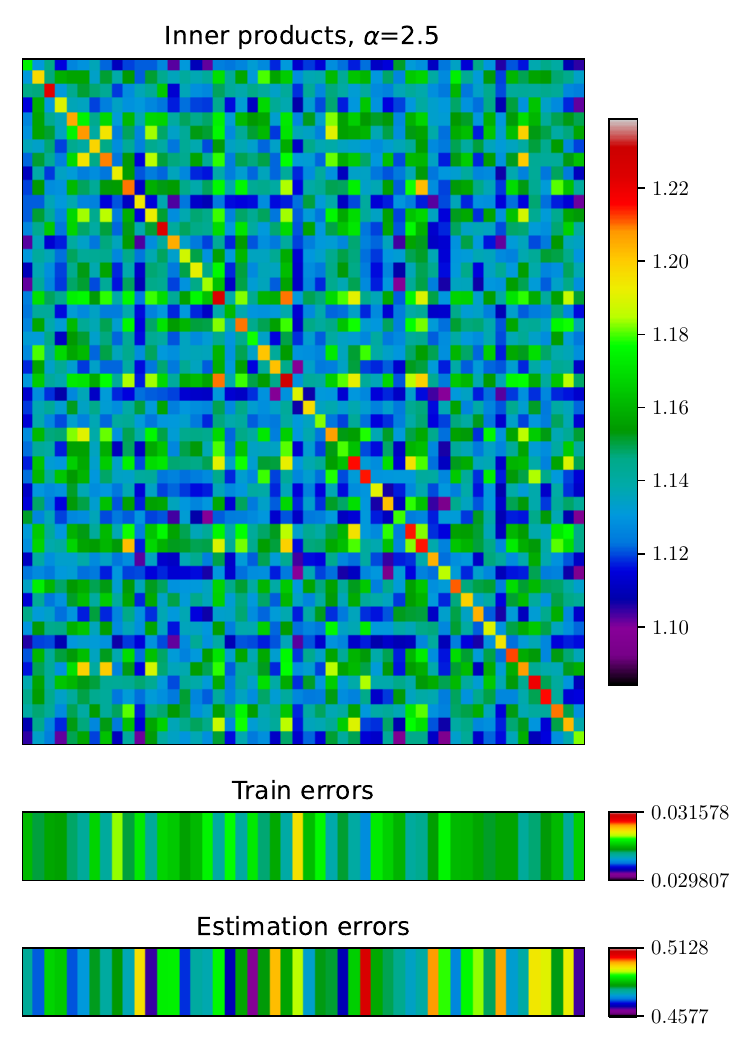} %
        \end{subfigure}%
        \hfill 
        \begin{subfigure}[t]{0.245\textwidth} 
            \centering
            \includegraphics[width=\textwidth,height=1.18\textwidth]{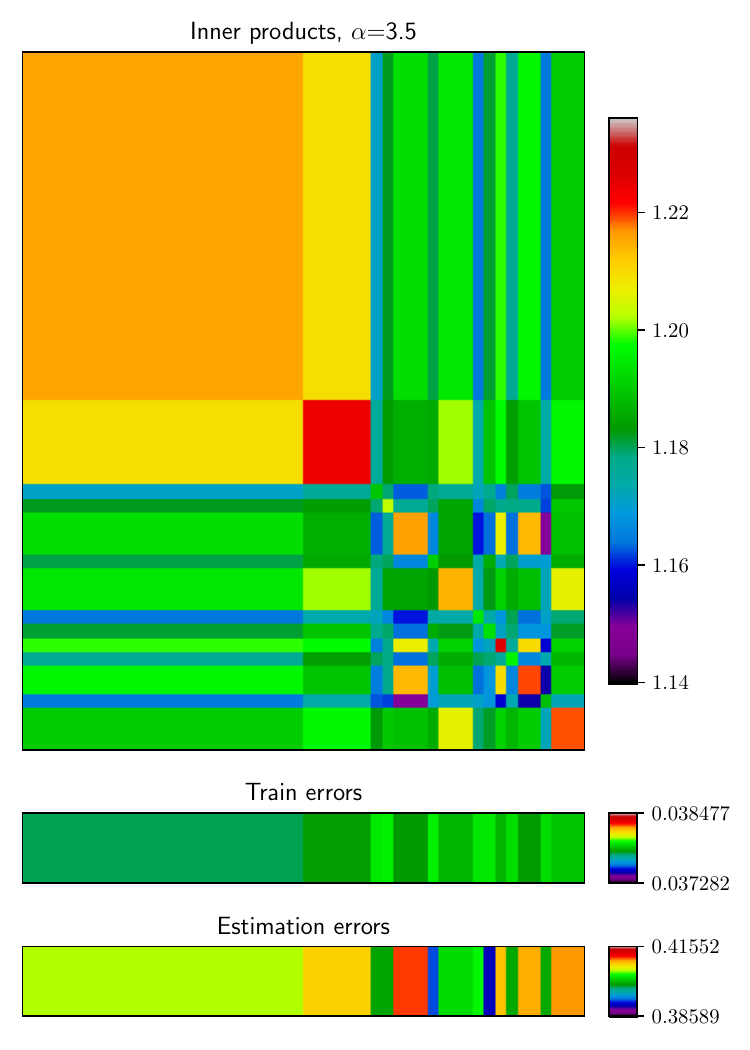} %
        \end{subfigure}%
        \hfill 
        \begin{subfigure}[t]{0.245\textwidth} 
            \centering
            \includegraphics[width=\textwidth,height=1.18\textwidth]{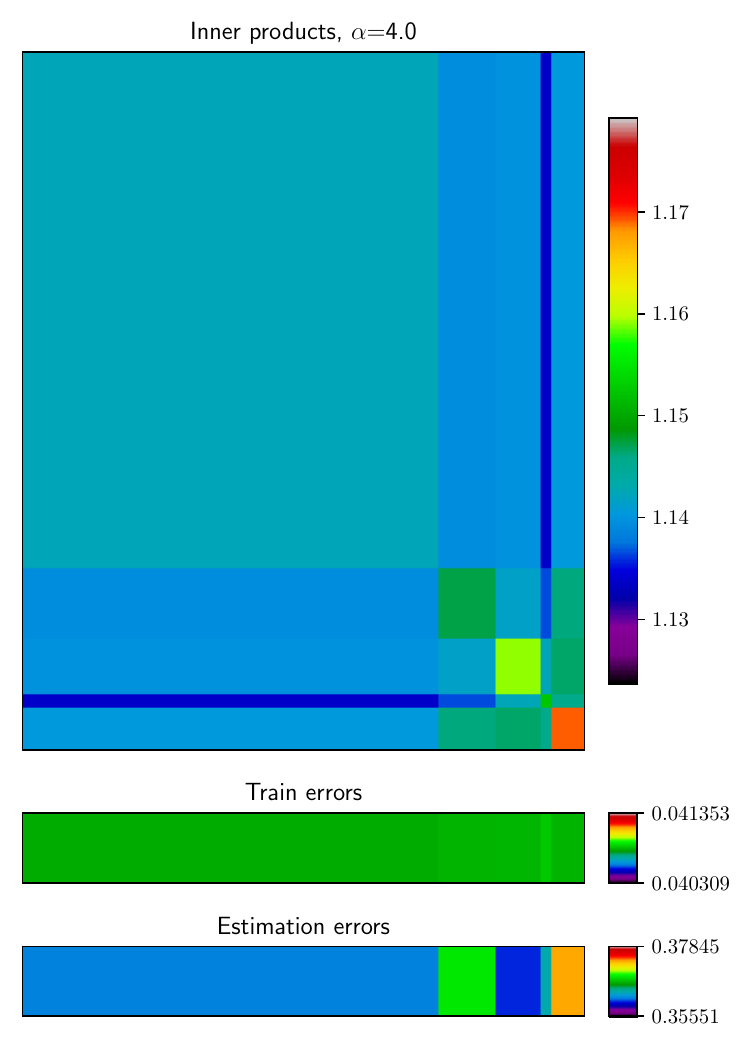} %
        \end{subfigure}%
        \hfill 
        \begin{subfigure}[t]{0.245\textwidth} 
            \centering
            \includegraphics[width=\textwidth,height=1.18\textwidth]{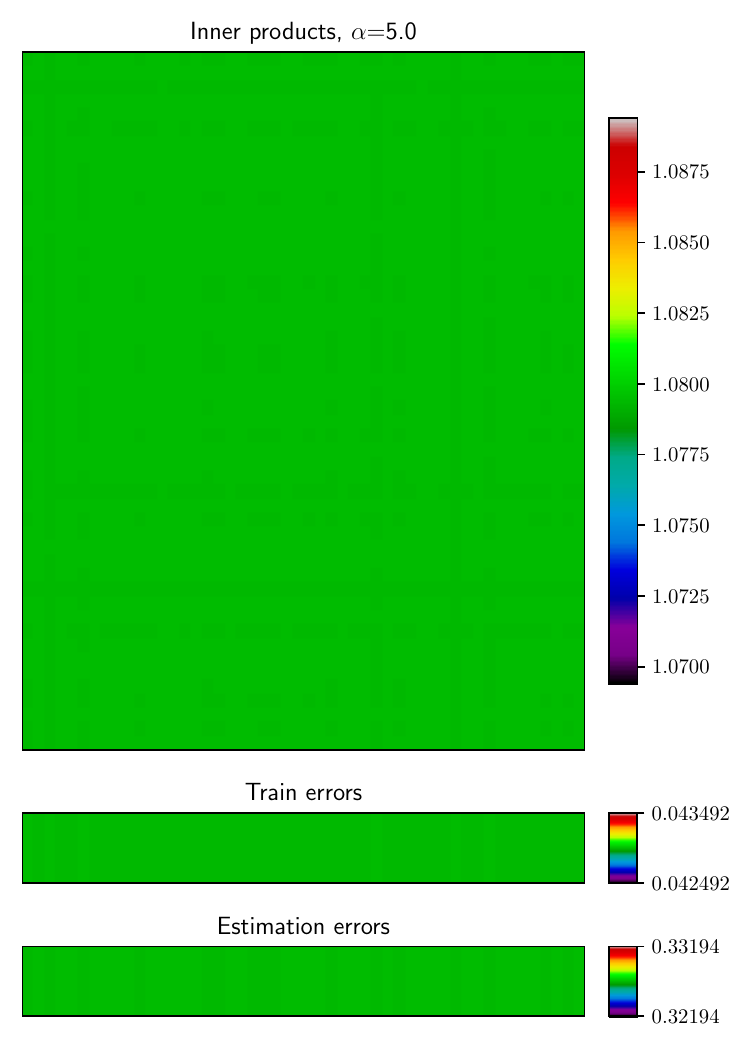} %
        \end{subfigure}%
        
        \caption{$M=50$ estimators $\hat\btheta^\up{1},\dots,\hat\btheta^\up{M}$ obtained from different random initialization of gradient descent on the same data $(\bX,\bw)$. The plots show the inner product  matrix $(\hat\btheta^{\up{i}\sT}\hat\btheta^{\up{j}})_{i,j \le M}$, along with the train error vector $(\hat R_n(\hat\btheta^\up{i}))_{i\in[M]}$ and the estimation error vector $(\|\hat \btheta^\up{i} - \btheta_0\|_2)_{i\in[M]}.$
        We take $d=1000$. The plots show one trial (i.e., one realization of ($\bX,\bw$)). 
        Note that the color scale is different for different values of $\alpha$.
        }
        \label{fig:heat_plot_transition}
    \end{figure}

In other words, the observation that $E_{3}(\alpha)\uparrow 1$ as 
$\alpha\downarrow \alpha_{\str}$ implies that the lower edge of the 
limiting spectral distribution of the Hessian 
vanishes as $\alpha\downarrow \alpha_{\str}$.
In order to see this, note that Lemma~\ref{lemma:G_to_s_star} below 
asserts that for a given $\nu$,
\begin{align}\label{eq:ZetaMin}
    \zeta_{\min}(\alpha;\nu)&:=\inf \supp(\mu_\star(\nu)) = -\inf_{s \in (0, \sfL_{\min}^{-1})} G_\ST(s;\alpha,\nu),
    \\
    G_\ST(s;\alpha,\nu) &:= \frac{1}{\alpha s} -  \E_\nu\left[\left(\frac{\ell''(v;v_0,w)}{1 + s \ell''(v;v_0,w)}\right)\right].
    \nonumber
\end{align}
(This statement is more general than the Tukey regression setting, hence the more general notation for $\ell$.)
Note that our theory at the moment does not rule out the possibility that 
the Hessian has $o(n)$ eigenvalues below $\zeta_{\min}(\alpha;\nu)$, although we do not expect this to be the case
for $\alpha>\alpha_{\str}$.

Figure~\ref{fig:FP_hessian} compares this prediction with numerical simulations.
We minimize the empirical risk via Gradient Descent (GD), compute the spectrum of the Hessian and 
display its support 
for  several values of $\alpha$ at $\SNR=\SNR_\star$ (the details of the GD scheme used are in Appendix~\ref{sec:numerical_details}).
We compare this  with $\zeta_{\min}(\alpha; \nu^\star)$ 
computed from Eq.~\eqref{eq:ZetaMin}, with 
$\nu^\star$ obtained by evaluating Eq.~\eqref{eq:NuoptRegr} the fixed point of iteration Eq.~\eqref{eq:NumScheme}. 
 We observe that ---as predicted--- the minimum eigenvalue of the Hessian becomes very close to 0 as 
 $\alpha\downarrow \alpha_{\str}$.

\paragraph{Rate function plots (Figure~\ref{fig:rate_function_plots}).}
Next, we show plots of the rate function for the expected number of local minima in local in 
Tukey regression, whose computation is enabled the finite dimensional formula in the second item of Theorem~\ref{thm:fin_dim_var_formula}.

Recall  the parameterization of the rate function $\Phi_{\tuk}(\rho;\iota)$ appearing in Corollary~\ref{prop:tukey}. Once again, in what follows, we choose the $\SNR=\SNR_\star$ as defined previously.
Figures~\ref{fig:rate_func_1} and~\ref{fig:rate_func_2} display the function 
\begin{equation}
\label{eq:phi_infty_def}
    \Phi_\infty(\rho) := 
    \inf_{\iota \in (0, \kappa^2/ 6)}\Phi_{\tuk}(\rho;\iota_\infty)
\end{equation}
resulting from solving the formula given in Corollary~\ref{prop:tukey} numerically. This is shown for several values of $\alpha$.
Recall that $\kappa^2/6$ is the maximal value of the empirical risk in this setting, so $\Phi_\infty$ corresponds to imposing no constraint on the value of empirical risk at the minimizer.
In Figure~\ref{fig:rate_func_1}, we see that for large values of 
$\alpha$, the rate function is non-negative for $\rho < \sfA_R$ 
(for some constant $\sfA_R$) but becomes negative for larger values of $\rho$.
This is a consequence of the fact that the the Tukey loss is convex near
the origin but becomes nonconvex and flat far from the origin.

The fact that $\Phi_{\infty}(\rho)$ is negative for large 
values of $\rho$ not inconsistent with Theorem~\ref{cor:robust_regression},
which only guarantees  non-negativity for $\alpha>\alpha_\star(\sfA_R).$

Figure~\ref{fig:rate_func_2} zooms in the interval of $\rho$ where
$\Phi \ge 0$ and $\Phi=0$  at a single point, while being locally convex around this point.
The threshold for this local convexity of the rate function  seems to occur in the vicinity of $\alpha =5.0$,
 in correspondence to the stability threshold discussed earlier (dashed-line in Figure~\ref{fig:FP_eqs}).

Let
\begin{equation}
\label{eq:iota_0}
    \iota_0(\alpha) := \sup\Big\{\iota \in [0, \kappa^2/6]:  
    \inf_{\rho \in (0,\infty)} \Phi_{\tuk}(\rho;\iota) >0 \Big\}.
\end{equation}
Informally, for a given $\alpha$, $\iota_0(\alpha)$ is the largest 
value of the training loss at which $\Phi_{\tuk}$ is non-negative for all $\rho$.
An immediate consequence of Corollary~\ref{prop:tukey} is that
the, for any $\eps>0$ there exists $c(\eps)>0$ such that, 
for all $n$ large enough,
\begin{align}
\P\Big(\min_{\btheta}\hR_n(\btheta) \le \iota_0(\alpha) -\eps \Big)\le 
e^{-c(\eps)n}\, .
\end{align}
Figure~\ref{fig:rate_func_3} shows the plot of 
\begin{equation}
\label{eq:phi_0_def}
    \Phi_0(\rho) := \inf_{\iota \in (0, \iota_0(\alpha))} \Phi_{\tuk}(\rho ;\iota)
\end{equation}
as a function of $\rho$, for different values of $\alpha$.
This constraint forces $\Phi_0(\rho)$ to be non-negative for larger
 values of $\rho$ (compare to Figure~\ref{fig:rate_func_1}).
Notice that, above the trivialization threshold $\alpha_{\str}$,
the location of the minimizing values of $\rho$ does not change when 
imposing the constraint on the training loss.  

\paragraph{Qualitative picture of the landscape.}
Let us pause for a moment to discuss the implications of the above findings.
For $\alpha$ sufficiently large, a separate argument (see Appendix~\ref{sec:numerical_details})
can be used to prove 
that there exists at least one  local minimizer
$\hbtheta$ such that $\|\hbtheta-\btheta_0\|\le \sfA_R$,
 and hence we can identify $(\rho^\star, \iota^\star)$ of Corollary \ref{prop:tukey} as the 
 asymptotic values of the estimation error
  and training loss of any such local minimizer (in fact we expect such 
  minimizer to be unique, but we do not formally prove it).
  Fixing a small $\Delta>0$, the  rate function 
  $\inf_{\iota \in (0, \iota_{\star} + \Delta] }\Phi_{\tuk}(\rho; \iota)$ is positive for all $\rho$
  except at the point $\rho^\star$
  (this corresponds to considering only local minima with 
  $\hR_n(\hbtheta)\le \iota_{\star}+\Delta$).
 This implies that
  the local minimizers  with loss/error $(\rho^\star, \iota^\star)$ are in fact 
  global minimizers,
  no other local minimizers are present in a large ball of radius $\sfA_R$,
  and outside this ball there are possibly other local minimizers albeit 
  with significantly larger train loss.

\paragraph{Predictions on estimation error and train loss 
(Figure~\ref{fig:preds_full}).}
In Figure \ref{fig:preds_full}, we compare our theoretical predictions for
 estimation error and train loss $(\rho^{\star}, \iota^\star)$ with the 
results of numerical experiments using GD.  We use $\SNR=\SNR_\star$ and vary 
$\alpha$ both over $\alpha > \alpha_{\str}(\SNR_\star)$ and 
over $\alpha < \alpha_{\str}(\SNR_\star)$
(recall that $(\alpha_{\str}(\SNR_{\star})\approx 5$).
We compare our the predictions with the 
result of running gradient descent.
Letting $\hbtheta_{\GD}$ denote the GD output, estimation error and
train loss are given, respectively, by $\|\hbtheta_{\GD}-\btheta_0\|_2$ 
and $\hR_n(\hbtheta_{\GD})$. These are labeled as \textsf{GD} in the figure.

For $\alpha > \alpha_{\str}$, predictions for the estimation and training 
errors are derived from Theorem \ref{cor:robust_regression}
using $(s^{\star},\rho^{\star})$ computed as fixed points of the 
recursion \eqref{eq:NumScheme}.
The corresponding theoretical predictions are
$\rho^{\star}$ and $\E_{\nu^\star}[\ell(v^{\star})]$ for the 
estimation error and training error respectively. These are labeled as 
\textsf{FP Eqs} in the figure.

For $\alpha < \alpha_{\str}$, the prediction for the estimation error 
are computed as
\begin{equation}
\label{eq:r_0_iota_0}
    \rho^{\star}(\alpha) := \argmin_{\rho > 0} \Phi_{0}(\rho),
\end{equation}
with $\Phi_0$ defined in Eq.~\eqref{eq:phi_0_def},
and those for the training error are given by $\iota_0(\alpha)$ of 
Eq.~\eqref{eq:iota_0}. These are labeled as \textsf{Rate Function} 
in the figure. 

We emphasize that the predictions for $\alpha < \alpha_{\str}$
are not expect to be asymptotically exact, but only a good approximations:
the expected number of local minima is exponentially large (even after constraining $\|\btheta-\btheta_0\|\le \sfA_R$)
and dominated by rare realizations of $\bX$ in this case. 
In contrast, results for $\alpha > \alpha_{\str}$ are likely to be exact (in the proportional asymptotics),
and indeed we prove they are  for all $\alpha$
large enough.

\noindent
\paragraph{Gradient descent stability experiments (Figures~\ref{fig:PhaseDiagram} and~\ref{fig:heat_plot_transition})}
Finally, to gain some qualitative understanding of the behavior of landscape above and below the threshold $\alpha_{\str},$ we perform the following experiment.
For a given $(\bX,\bw,\btheta_0)$, we run gradient descent from $M$ different random initializations on the unit sphere $\S^{d-1}(1)$ to produce estimations $\hat\btheta^\up{1},\dots,\hat\btheta^\up{M}.$ (In these experiments $M=30$)

In Figure~\ref{fig:heat_plot_transition}, we plot, for values of $\alpha\in\{2.5,3.5,4.0,5.0\}$, the $M\times M$ matrix of inner products $(\hat\btheta^{\up{i}\sT}\hat\btheta^\up{j})_{i,j\in[M]}$ along with the vectors $(\hat R_n(\hat\btheta^\up{i}))_{i\in[M]}$ and $(\|\hat\btheta^\up{i} - \btheta_0\|_2)_{i\in[M]}.$ We see that as $\alpha$ increases towards $\alpha_{\str}(\SNR_\star),$ the rank of the matrix decreases, and the minimizers starting from different initializations coincide.

In Figure~\ref{fig:PhaseDiagram} of the introduction, we
compare the threshold at which the gradient descent solutions from different initializations begin to coincide, with the threshold $\alpha_{\str}(\SNR)$.
Namely, in Figure~\ref{fig:PhaseDiagram_1}, we plot $\max_{i,j \in[M]}\|\hat\btheta^\up{i} -\hat\btheta^\up{j}\|_2$  in the $(\alpha,\SNR)$-plane, averaged over $N$ different realizations of $(\bX,\bw).$ 
In Figure~\ref{fig:PhaseDiagram_2}, we plot the number of clusters that the different solutions $\hat\btheta^\up{i}$ form, computed by thresholding the normalized distance by $\eps = 10^{-3}$. The result shown is the average of the number of clusters over $N$ different realizations of $(\bX,\bw).$
Finally, overlayed on both figures is the curve of $\alpha = \alpha_{\str}(\SNR)$ from Figure~\ref{fig:FP_region}.

\section{Proofs}

We will now prove the main results of the paper beginning with Theorem~\ref{thm:fin_dim_var_formula}. We will leave some of the more technical details to the appendix.
\subsection{Proof of the variational upper bound: Theorem~\ref{thm:fin_dim_var_formula}}

\subsubsection{The general lower bound on the rate function}
We state the result of~\cite{asgari2025local} (adapted to our setting) which  we use as a starting point for our proofs. 

\begin{theorem}[Theorem 1 of~\cite{asgari2025local}]
\label{thm:general}
In the setting of Theorem~\ref{thm:fin_dim_var_formula_k=1}, with the same definitions made there, we have for any $\cuA,\cuB$ open in the $W_2$ topology and any $\sfA_R,\sfa_L, c_0>0,$
\begin{align}
 \lim_{\delta\to0}\lim_{n\to\infty}\frac{1}{n}\log\E\left[Z_n(\cuA,\cuB,\sfA_R,\sfa_L) \one_{\Omega_\delta \cap\tilde\Omega_0 \cap \tilde\Omega_1(c_0)}\right]
 &\le
 -\hspace{-5mm}
\inf_{\substack{\mu \in \cuA\cap\cuV_{\mupart}\cap\tilde\cuG_\mupart 
} }
\left\{
\Phi_\mupart(\mu) +
\inf_{
\nu \in\cuB \cap \cuV_\nupart(\bR(\mu))\cap \tilde\cuG_\nupart
}
  \tilde\Phi_{\nupart}(\nu,\bR(\mu)) 
\right\}
\end{align} 
where
\begin{equation}
    \tilde\Omega_0 := \{ \|\bX\bSigma^{-1/2}\|_\op \le 2\sqrt{n}( 1 + \alpha^{-1/2}) \},
\end{equation}
\begin{equation}
    \tilde\Omega_1(c_0) :=
    \{\forall \btheta\in\cG_n\;\textrm{with}\; \grad\hat R_n(\btheta) = \bzero,\; \sigma_{\min}([\btheta,\btheta_0]) > c_0\},
\end{equation}
\begin{equation}
    \tilde\cuG_\mupart := \cuG_{\mupart} \cap \{\mu : \bR(\mu) \succ c_0\},\quad\quad
    \tilde\cuG_\nupart := \cuG_{\nupart} \cap\{\nu : \E_\nu[(\bv,\bv_0)(\bv,\bv_0)^\sT]\prec \sfA_R^2(1 + \alpha^{-1/2})^2\},
\end{equation}
$\Phi_{\mupart}$ was defined in Eq.~\eqref{eq:phi_A_phi_B_defs}, and 
   \begin{align}
   \label{eq:tilde_phi_B}
    &\tilde\Phi_{\nupart}(\nu,\bR)  := 
    -\frac{k}{\alpha}\int  \log(\zeta + \lambda) \mu_\star(\nu)(\de\zeta)
    +  \frac1{2\alpha}\log \left(
    \frac{
    \E_\nu[\ell'(v;v_0,w)^2 ]}{
\alpha\, e \,\schur(\bR,r_{00})
    }\right) 
    +\KL(\nu_{\cdot|w} \| \cN(\bzero ,\bR ))
   \end{align}

\end{theorem}

By comparing the statement of this theorem to the statement of Item~\textit{1} of Theorem~\ref{thm:fin_dim_var_formula_k=1}, we see that the main difference is the term involving the logarithmic potential. In the convex case,~\cite{asgari2025local} proves a variational principle for this logarithmic potential term which we recall now.
It will be useful to first define the following quantities:
recalling $L_0(\nu)$ defined in the statement of the theorem, for $z\in \R\setminus\supp(\mu_\star(\nu))$ and $s\in(0,L_0^{-1}(\nu))$, let 
\begin{equation}
    K_\ST(s;z) := K_\ST(s;z,\nu) := -\alpha z s + \alpha \E_\nu\left[\log(1 + s  \ell''(v,v_0,w)) \right] - \log s - (\log \alpha + 1).
\end{equation}
The following lemma gives the variational formula for the convex setting.

\begin{lemma}[Lemma~7 of ~\cite{asgari2025local}]
\label{lemma:log_pot_convex}
Fix $\nu\in\cuP(\R^3)$ and assume that $\ell$ is convex on the support of $\nu$, i.e., $L_0(\nu) = 0$.
Then 
\begin{equation}
    \int \log(\zeta + \lambda) \mu_\star(\nu)(\de \zeta) \le \inf_{s > 0}
    K_\ST(s;-\lambda, \nu).
\end{equation}
\end{lemma}

In the next section, we extend this to the present non-convex setting.

\subsubsection{The variational bound for the logarithmic potential}

Throughout this sub-section, we fix a measure $\nu \in\cP(\R^{3})$
such that $\inf\supp(\mu_{\star}(\nu))\ge 0$. 
Assume $L_0(\nu) >0$, so that $\ell$ is not convex on the support of $\nu$.
Define for $s\in(0,L_0(\nu)^{-1})$,
\begin{equation}
    G_\ST(s) := G_\ST(s;\nu) := \frac1{\alpha s } - \E_\nu\left[ \frac{\ell''(v;v_0,w)}{1 + s\, \ell''(v;v_0,w)} \right].
\end{equation}

We will extend the definition of $K_\ST$ and $G_\ST$ to $L_0(\nu)$ by continuity,
\begin{equation}
K_\ST(L_{0}(\nu)^{-1}; z, \nu) := \lim_{s\to L_{0}(\nu)^{-1}} K_\ST(s;z,\nu),
\end{equation}
and similarly for $G_\ST$.
 Note that the limit always exist by monotonicity, albeit
 it is potentially infinite.

The following lemma summarizes the relevant properties of $G_\ST$ in the non-convex setting. Specifically, it relates
 the properties of the function $G_\ST$ to the Stieltjes transform $s_{\star}.$ 
\begin{lemma}
\label{lemma:G_to_s_star}
Fix $\nu\in\cuP(\R^{3})$ such that $L_0=L_0(\nu) >0$.
    The function $G_\ST(s)$ has a unique global minimizer $s_{\min} :=s_{\min}(\nu)$ in  
    $(0,L_{0}^{-1}]$ satisfying
    \begin{equation}
        \zeta_{\min}(\nu) := \inf\supp(\mu_\star(\nu)) = - G_\ST(s_{\min}).
    \end{equation}
    Furthermore, for
        $z\in(-\infty, -G_\ST(s_{\min })]$, $s_{\star }(z)$ is finite, and is the unique solution to 
        \begin{equation}
            G_\ST(s) = - z, \quad\quad s \in (0,L_{0}^{-1}],\quad\quad\; G_\ST'(s) \le 0.
        \end{equation}
\end{lemma}
 The result can be derived from the classical result of~\cite{silverstein1995analysis}. We include a direct proof in Appendix~\ref{sec:proof_of_lemma_G} for completeness.
 With this result, we can then pass to the following variational formula for the logarithmic potential.

\begin{lemma}
\label{lemma:variational_form_logdet}
Fix $\nu \in\cuP(\R^{3})$ with $L_0=L_{0}(\nu) > 0.$
Work in the setting of Lemma~\ref{lemma:G_to_s_star}
with
\begin{equation}
    s_{\star}(z) \le s_{\min} \le L_{0}^{-1}\quad\quad \textrm{for}\quad z\in (-\infty, \zeta_{\min}].
\end{equation}
Then for any $z\in(-\infty, \zeta_{\min}]:$
\begin{enumerate}
\item
\label{item:lemma_var_item_1}
We have%
\begin{equation}
\label{eq:K_is_log_pot}
      K_\ST(s_\star(z);z)
    =\int \log(\zeta - z) \mu_\star(\de\zeta).
\end{equation}
    \item 
\label{item:lemma_var_item_2}
    $s_\star(z)$ is the unique solution to 
    \begin{equation}
    \partial_s K_\ST(s;z) = 0, \;\; \partial^2_s K_\ST(s;z) \ge 0,\;\; s\in (0,L_{0}^{-1}].
    \end{equation}
    \item  
\label{item:lemma_var_item_3}
We have
    \begin{equation}
        K_\ST(s_\star(z),z) = \min\left\{ K_\ST(s,z) : s\in
         (0,L_0^{-1}]
        ,\; \E_\nu\left[ \left(\frac{s\ell''(v;v_0,w)}{1 + s\ell''(v;v_0,w)}\right)^2\right] \le \frac1{\alpha}\right\}.
    \end{equation}    
\end{enumerate}
\end{lemma}

The proof is once again deferred, and can be found in Appendix~\ref{sec:pf_lemma_variational_form_logdet}.
We now give a proof of Theorem~\ref{thm:fin_dim_var_formula}.
\subsubsection{Proof of Theorem~\ref{thm:fin_dim_var_formula_k=1}}
\noindent
\textit{Proof of Item~\textit{\ref{item:thm_1_item_1}}:}
Fix $\mu\in\cuV_\mupart$, and $\nu\in\cuV_\nupart(\bR(\mu))$, 
and note that $\nu$ must satisfy $\inf \supp (\mu_{\star}(\nu)) \ge -\lambda$.
So if $L_0(\nu) >0$, then Items~\textit{\ref{item:lemma_var_item_1}} and~\textit{\ref{item:lemma_var_item_3}} of Lemma~\ref{lemma:variational_form_logdet} with $z = -\lambda$ imply
\begin{equation}
\label{eq:upper_bound_var_fixed_j}
    \int\log(\zeta + \lambda)  \mu_{\star}(\nu)(\de \zeta) 
    \le  \inf_{s \in\cS_0(\nu)} K_\ST(s, -\lambda;\nu),
\end{equation}
for all $\nu$. Meanwhile, if $L_0(\nu) = 0$,
then Lemma~\ref{lemma:log_pot_convex} gives
 Eq.~\eqref{eq:upper_bound_var_fixed_j} as well (in-fact, it holds with equality if $L_0(\nu) > 0)$. 
So $\tilde \Phi_\nupart$ in Theorem~\ref{thm:general} is lower bounded as
\begin{equation}
    \tilde \Phi_\nupart(\nu,\bR) \ge \sup_{s\in\cS_0(\nu)} \Phi_\nupart(\nu,\bR,s).
\end{equation}

Now clearly $\tilde\Omega_0$ in the statement of Theorem~\ref{thm:general} is a high-probability set. Meanwhile, by Lemma~41 of~\cite{asgari2025local}, there exists $c_0>0$ so that $\tilde\Omega_1(c_0)$ is also a high-probability set, so that we can set the high-probability set in the statement of Item~\textit{\ref{item:thm_1_item_1}} as $\Omega'_0 := \tilde\Omega_0 \cap\tilde\Omega_1(c_0)$ for this choice of $c_0$.
Finally, since $\cuG_\nupart \supseteq \tilde\cuG_\nupart,\cuG_\mupart \supseteq \tilde\cuG_{\mupart},$ an application of Theorem~\ref{thm:general} now yields the desired result. \qed\newline

\noindent
\textit{Proof of Item~\textit{\ref{item:thm_1_item_2}}:}
Fix $\nu\in\cuP(\R^3)$ with $\inf\supp(\mu_\star(\nu)) > -\lambda.$ 
If $L_0(\nu) > 0$,
then by Lemma~\ref{lemma:G_to_s_star}
 $s_{\star}(-\lambda;\nu)$ uniquely solves
\begin{equation}
\label{eq:s_star_exp}
    \E_\nu\left[\frac{s \ell''}{1 + s\ell''}\right] = \frac1\alpha + s\lambda,\quad\quad
    \E_\nu\left[\left(\frac{s \ell''}{1 + s\ell''}\right)^2\right] \le \frac1\alpha
\end{equation}
for $s \in (0,L_{0}^{-1}(\nu)]$.
Meanwhile, if $L_0(\nu) = 0$, then again $s_\star(-\lambda;\nu)$ solves the above uniquely for $s\in(0,\infty)$ by Lemma~37 of~\cite{asgari2025local}.

So the function $\nu \mapsto s_{\star}(-\lambda;\nu)$ is well-defined on such measures $\nu$, and by 
Lemma~\ref{lemma:variational_form_logdet}
and Lemma~37 of~\cite{asgari2025local},
we have
\begin{equation}
\tilde \Phi_\nupart(\nu,\bR) = \Phi_{\nupart}(\nu,\bR,s_\star(-\lambda;\nu)),
\end{equation}
where $\tilde\Phi_\nupart$ was defined in Theorem~\ref{thm:general}.

Now, it is clear that the sets $\cuA_\cR,\cuB_\cL$ are open in the $W_2$ topology by the Pseudo-Lipschitz assumption on the functions involved in their definitions. So as in the proof of 
the first item above, by working on the appropriate high-probability 
sets, we can apply Theorem~\ref{thm:general} to obtain an upper-bound on the left-hand side of Eq.~\eqref{eq:computation_lb} given by
\begin{equation}
     -\inf_{\mu \in\cuV_\mupart\cap\cuA_\cR}
     \inf_{\nu \in\cuV_\nupart(\bR(\mu)) \cap \cuB_\cL} 
    \left\{ 
    \Phi_{\mupart}(\mu) + \Phi_{\nupart}\big(\nu ,\bR(\mu), s_\star(-\lambda;\nu)\big)
    \right\}.
\end{equation}

Now introducing the notation $g(\nu) := \E_\nu[\ell'(v;v_0,w)^2]$,  
we note that we can perform the minimization above as
\begin{equation}
   \inf_{\substack{\bR \in\cR \\
    r_{00} = \int t_0^2  \mu_0 ( \de t_0)
   }} 
   \inf_{L \in\cL} \Phi_0(\bR,L)
\end{equation}
where
\begin{align}
\Phi_0(\bR,L) :=
\inf_{\substack{g > \sfa_L^2 \\
s >0
}}
\inf_{\substack{\mu \in \cuV_\mupart  \\
\bR(\mu) = \bR
}}
\inf_{ \substack{\nu \in \cuV_\nupart(\bR) \\ 
\E_\nu[\psi] = L\\
g(\nu) = g\\
s_{\star}(-\lambda;\nu) = s
}}
\Bigg\{
\Phi_\mupart(\mu)
 &- \lambda s
 + \frac1{2\alpha} \log\det\left(\frac{\alpha\,e\, s^2\, g}{
\schur(\bR,r_{00})
 }\right)\nonumber\\
  &
 - \E_\nu[\log(1 + s \ell''(v;v_0,w)] 
  + \KL(\nu_{\cdot |w} \| \cN(\bzero,\bR))
\Bigg\}
\end{align}
with the infimum over $\nu$ defined to be $+\infty$ whenever
 there is no $\nu$ satisfying the constraints specified by $g$ and $s$. 

Clearly, the minimum over $\mu$ of $\Phi_{\mupart}(\mu)$ is equal 
to zero since for any $\bR \succ\bzero$, we can choose the minimizing 
$\mu^\star$ so that $\mu^\star_{\cdot |t_0} = \cN(r_{10}(\mu)^\sT
 r_{00}^{-1}t_0, \schur(\bR(\mu),r_{00})$ to achieve the zero value
  of the divergence in the definition of $\Phi_{\mupart}(\mu),$ while
 $\mu^\star_{(t_0)}=\mu_0$ (it's easy to check that the resulting 
 $\mu^\star$ will have $\bR(\mu^\star) = \bR$ and will therefore be
  feasible).

What's left now is to minimize terms involving $\nu$.
We first simplify the constraints on $\nu$ appearing above:
First note that for any $\nu$ not satisfying the support constraint $\mu_{\star}((-\infty,-\lambda)) = 0$ appearing in $\cuV_{\nupart}$, Lemma~\ref{lemma:G_to_s_star} implies that there is no $s$ satisfying the conditions in Eq.~\eqref{eq:s_star_exp}, and as a result, we can drop this support constraint (in fact we can always drop this constraint to obtain a lower bound). The minimization over $\nu$ then is equivalent to
\begin{equation}
\label{eq:isolated_min_over_nu}
    \inf_{
    \substack{
    \nu \in \cuV_0\\
    \inf\supp((1 + s\ell'')_{\# \nu}) >0  \\
    } } 
    \bigg\{
 - \E_\nu[\log(1 + s \ell''(v;v_0,w)] +
    \KL(\nu_{\cdot |w } \| \cN(\bzero,\bR))\bigg\},
\end{equation}
where
\begin{align}
\nonumber
    \cuV_{0} := \Big\{\nu \in\cuP(\R^{3}) :
    \nu_{(w)} = \P_w,\;
    \E[(v,v_0)\ell'(v;v_0,w)] &= - \lambda(r_{11},r_{00}),\;
    \E_\nu[\psi] = L,\;
     \E_\nu[\ell'(v;v_0,w)^2] = g,\;\\
    &\E_\nu\Big[\frac{s \ell''}{1 + s\ell''}\Big] = \frac1\alpha,\;
    \E_\nu\bigg[\Big(\frac{s \ell''}{1 + s\ell''}\Big)^2\bigg] \le \frac1\alpha
    \Big\}.
\end{align}

What is left now is to perform this minimization over $\nu$ and introduce the Lagrange parameters for the linear constraints appearing in the definition $\cuV_{0}$. Direct simplifications then give the definition of $\Phi_{\fin}$ in the theorem statement.
\qed

\subsection{Proof of trivialization: Theorem~\ref{thm:trivialization_k=1}}

We now move on to the proof of Theorem~\ref{thm:trivialization_k=1}.
Let us begin with the following lemma stating that the proximal operator in Definition~\ref{def:opt_FP_conds} is well defined.
\begin{lemma}
\label{lemma:prox_regularity}
Under Assumption~\ref{ass:loss},
for all $(x_0,w) \in\R\times \supp(\P_w)$ and $s \in (0, \sfL_{\min}^{-1})$, the map
\begin{equation}
     x\mapsto \Prox_{\ell(\cdot;x_0, w)}( x ; s)
\end{equation}
is a well-defined, invertible, differentiable function for all $x\in\R,$ with derivative
\begin{equation}
\frac{\partial}{\partial x} \Prox_{\ell(\cdot;x_0,w)}(x;s)= 
\frac{1}{1 + s \ell''(\Prox_{\ell(\cdot;x_0,w)}(x;s); x_0, w)}.
\end{equation}
Furthermore, it satisfies the stationarity condition
\begin{equation}
\label{eq:prox_stationarity}
   s \ell'(
   \Prox_{\ell(\cdot;x_0, w)}( x ; s))
    + \Prox_{\ell(\cdot;x_0, w)}( x ; s) = x.
\end{equation}
\end{lemma}
The proof is fairly straightforward, and included in Appendix~\ref{proof:prox_regularity} for completeness.

\subsubsection{
Proof of
Theorem~\ref{thm:trivialization_k=1}
}

\begin{proof}[Proof of
Item~\textit{\ref{item:thm_2_item_1}}
of Theorem~\ref{thm:trivialization_k=1}:]

It is not hard to check that Eq.~\eqref{eq:Phi_at_opt_is_0} holds
when $(\mu^\opt, \nu^\opt)$ satisfy Definition~\ref{def:opt_FP_conds}.
Evaluating $\Phi_\star(\mu^\opt,\nu^\opt)$ follows 
an argument  similar to the one in Section 7.1.2 of~\cite{asgari2025local}. 
In order to verify that $(\mu^\opt, \nu^\opt)$ are indeed feasible,
we need to check conditions \eqref{eq:cuV_mupart_def}, \eqref{eq:cuV_mupart_def},
as well as Eqs.~\eqref{eq:S-Set-Def} to \eqref{eq:cuG_nupart_def}.
The only condition of these ones
that is not straightforward to verify is 
$\inf\supp(\mu_{\star}(\nu^\opt)) \ge -\lambda.$ 
In order to check it, note that, by Lemma~\ref{lemma:G_to_s_star},
existence of $s^\opt$ satisfying Eq.~\eqref{eq:Stieltjis-Condition} implies 
that $s^\opt$ must be the Stieltjes transform of
$\mu_{\star}$ at $-\lambda$ and 
 that $\zeta_{\min}(\nu^\opt) \ge -\lambda$.
\end{proof}

For the second item in Theorem \ref{thm:trivialization_k=1}, we state the following lemma which allows us to rewrite $\Phi_\nupart$ into a more tractable form.
Its proof is deferred to Appendix~\ref{proof:prox_density}.
\begin{lemma}
\label{lemma:prox_density}
For $s\in(0,\sfL_{\min}^{-1}),\;\bR\succ\bzero$
consider the non-negative function 
\begin{equation}\label{eq:PiDensity}
       \pi_{\bR,s}(v|v_0,w) := \frac{(1 + \ell''(v;v_0,w))}{\sqrt{2\, \pi\, \schur(\bR,r_{00})}} 
       \exp\left\{ - \frac{\left(s \ell'(v;v_0,w) + v\right)^2}{2\, \schur(\bR,r_{00})}\right\}.
\end{equation}
\begin{enumerate}
    \item For  
    $(v_0,w)\in\R\times\supp(\P_w)$, 
    $\pi_{\bR,s}(v|v_0,w)$ is the conditional density of the random variable 
    \begin{equation}
    v :=(\Prox_{\ell(\cdot,v_0,w)}(g;s))
    \quad
    \textrm{where}\quad  (g,v_0,w) \sim \cN(\bzero_2, \bR) \otimes \P_w.
    \end{equation}
\item  For any $\bR\succ\bzero$, $\nu \in\cuV(\bR)$,
we have
\begin{align}
\label{eq:Phi_B_prox}
    \Phi_{\nupart}(\nu,\bR,s)
    &= 
    \KL(\nu_{\cdot|v_0, w} \| \pi_{\bR,s}(\cdot |v_0,w )) + 
    \KL(\nu_{v_0|w}\|  \cN(0,r_{00}))
    + \frac1{2\alpha} F_\nupart\left(
    \frac{\alpha \, s^2\, \E[\ell'(v;v_0,w)^2]}{\schur(\bR,r_{00})}
    \right) 
\end{align}
where
\begin{equation}
    F_\nupart(x) := \log(x) - x+ 1.
\end{equation}
\end{enumerate}
\end{lemma}

Next, we simplify the constraint set $\cS_0$.
Recall the definitions of $\alpha_0$ and $\tau_0$ in Eq.~\eqref{eq:alpha_0_hat_s_def}.
The following lemma will allow us to restrict $s$ to a set that is independent of $\nu$.
The proof is given in Appendix~\ref{proof:bars_bound_enough}.
\begin{lemma}
\label{lemma:bars_bound_enough}
For all $\nu\in\cuP(\R^{3})$, if $\alpha \ge\alpha_0$, we have 
\begin{equation}
   \frac{\tau_0}{ \sqrt{\alpha}} < \frac1{\sfL_{\min}}\quad\quad\textrm{and}\quad\quad
   \left(0, \frac{ \tau_0}{\sqrt{\alpha}}  \right) \subseteq \cS_0(\nu).
\end{equation}
\end{lemma}

With these in hand, we are ready to provide a proof for Item~\textit{\ref{item:thm_2_item_2}}.

\begin{proof}[Proof of
Item~\textit{\ref{item:thm_2_item_2}}
of Theorem~\ref{thm:trivialization_k=1}]
Fix $\mu \in\cuV_{\mupart},\nu\in\cuV_{\nupart}(\bR(\mu))$ satisfying $(\mu,\nu)\in\cuT$.
We suppress the dependence on $\mu$ in the notation for $\bR$ and its blocks in what follows.
The previous two lemmas imply
\begin{align}
\label{eq:lower_bound_Phi_triv_proof}
    \sup_{s\in{\cS_0}(\nu)} &\Phi_{\nupart}(\nu,\bR,s) 
    \stackrel{(a)}{\ge}
    \sup_{ s \in (0,\tau_0\,\alpha^{-1/2})} \Phi_{\nupart}(\nu,\bR,s)\\
    &\stackrel{(b)}{=}
    \sup_{s\in (0, \tau_0\, \alpha^{-1/2})}
    \left\{
    \KL(\nu_{\cdot|v_0, w} \| \pi_{\bR,s}(\cdot |v_0,w )) + 
    \KL(\nu_{v_0| w}\| \cN(0,r_{00}))
    + \frac1{2\alpha} 
    F_\nupart\left(
    \frac{\alpha \, s^2\, \E[\ell'(v;v_0,w)^2]}{\schur(\bR,r_{00})}
    \right) 
    \right\},
    \nonumber
\end{align}
where $(a)$ follows by Lemma~\ref{lemma:bars_bound_enough} and $(b)$ follows by
the second item of Lemma~\ref{lemma:prox_density}, whose applicability is justified for all $s \in \tau_0 \alpha^{-1/2}$ by Lemma~\ref{lemma:bars_bound_enough}.

Now note that $(\mu,\nu)\in\cuT$ guarantee that
\begin{equation}
    \hat s(\bR,\nu) := \left(\frac{\schur(\bR,\bR_{00})}{\alpha \E_\nu[\ell'^2]}\right)^{1/2} 
    \le 
    \frac{\tau_0}{\sqrt{\alpha}},
\end{equation}
so that $\hat s(\bR,\nu)$ is a feasible choice of $s$.
Estimating the lower bound of Eq.~\eqref{eq:lower_bound_Phi_triv_proof} by the value of the objective at this point of $s$ gives
\begin{equation}
   \sup_{s\in{\cS_0}(\nu)} \Phi_{\nupart}(\nu,\bR,s) 
   \ge
    \KL(\nu_{\cdot|v_0, w} \| \pi_{\bR,\hat s(\bR,\nu)}(\cdot |v_0,w )) + 
    \KL(\nu_{v_0| w}\| \cN(0,r_{00}))
\end{equation}
since $F_\nupart(1) = 0$.
So by the definition of $\Phi_\star(\mu,\nu)$ given in Theorem~\ref{thm:fin_dim_var_formula_k=1}, we have
\begin{equation}
   \Phi_\star(\mu,\nu) \ge  
  \frac1\alpha 
\KL\left(\mu_{\cdot|t_0}\| \cN(r_{10}(\mu) r_{00}^{-1} t_0,\schur(\bR,r_{00})\right)
+  
    \KL(\nu_{\cdot|v_0, w} \| \pi_{\bR,\hat s(\bR,\nu)}(\cdot |v_0,w )) + 
    \KL(\nu_{v_0| w}\| \cN(0,r_{00}))
\end{equation}
which is strictly positive for all $(\mu,\nu)\in\cuT$ 
with $\mu\in\cuV_{\mupart}, \nu\in\cuV_{\nupart}(\bR(\mu))$, unless 
\begin{equation}
    \mu_{(t_0)} = \mu_0,\quad\textrm{and for all}
    \quad t_0 \in\supp(\mu_0))\quad\textrm{we have}\quad
    \mu_{\cdot|t_0} = 
    \cN(r_{10}^\sT r_{00}^{-1} t_0,\schur(\bR,r_{00}),\label{eq:FirstPosUnless}
\end{equation}
and
\begin{equation}
\label{eq:nu_conditions_for_opt_k=1}
    \nu_{\cdot|v_0,w}(\de v) = \pi_{\bR,\hat s}(v |v_0,w) \de v\, ,
    \quad
    \nu_{\cdot | w}=\cN(0,r_{00}),
    \quad
    \nu_w = \P_w,\quad
    \E_{\nu}[\ell'(v;v_0,w))(v,v_0)] + \lambda (r_{11},r_{10}) = \bzero,
\end{equation}
with
\begin{equation}
\label{eq:hat_s}
    \hat s^2 =  \frac{\schur(\bR,r_{00})}{\alpha \, \E_{\nu}[\ell'(v;v_0,w)^2]}.
\end{equation}

We are left to show that, under Eqs.~\eqref{eq:FirstPosUnless}
to \eqref{eq:hat_s}, $\Phi_{\star}(\mu,\nu)=0$ implies the last condition
in Definition \ref{def:opt_FP_conds}, namely 
Eq.~\eqref{eq:Stieltjis-Condition}. 

By the properties of the proximal operator from Lemma~\ref{lemma:prox_density}, 
and recalling that $(v,w_0,w)$ is distributd as 
$(\Prox_{\ell(\cdot;g_0,w)}(g,\hat s),g_0,w)$ 
(under the ansatz \eqref{eq:nu_conditions_for_opt_k=1})
we have
\begin{align}
    \hat s\E_{\nu}[ \ell'(v;v_0,w)^2] &\stackrel{(a)}{=} 
    \E_0[g\,\ell'(\Prox_{\ell(\cdot;g_0,w)}(g; \hat s)] 
    - \E_{\nu}[v\ell'(v;v_0,w) ]\\
    &\stackrel{(b)}{=}
    \schur(\bR,r_{00}) \E\left[\frac{\ell''(v,v_0,w)}{1 + \hat s \ell''(v;v_0,w)}\right]
    +\E[r_{10}^\sT r_{00}^{-1}v_0\,\ell'(v;v_0,w)] 
    - \E[\ell'(v;v_0,w) v]
    \\
    &\stackrel{(c)}{=}
    \schur(\bR^\opt,r_{00}) \E\left[\frac{\ell''(v;v_0,w)}{1 + \hat s \ell''(v;v_0,w)}\right]
    +\lambda\,\left(-r_{10}^{\sT}r_{00}^{-1}r_{10}
    +  r_{11}\right)
\end{align}
where $(a)$ follows by the stationarity condition in Eq.~\eqref{eq:prox_stationarity}
with $\E_0$ denoting the expectation under $(g,g_0,w)\sim\cN(\bzero,\bR) \otimes \P_w$, 
$(b)$ follows by Gaussian integration by parts and the derivative formula in Lemma~\ref{lemma:prox_density}, and $(c)$ follows by Eq.~\eqref{eq:nu_conditions_for_opt_k=1}.
This 
shows that Eq.~\eqref{eq:hat_s} is equivalent to
\begin{equation}
    \E\left[\frac{\ell''(v;v_0,w)}{1 + \hat s \ell''(v;v_0,w)}\right] + \lambda = \frac1{\alpha\, \hat s}.
\end{equation}
Finally, $\hat s \le \tau_0 \alpha^{-1/2}$ implies---by Lemmas~\ref{lemma:bars_bound_enough}---that $\hat s \in \cS_0(\nu)$.
So we've shown that if $\Phi_\star(\mu,\nu) =0$, then $(\mu,\nu)$ satisfy Definition~\ref{def:opt_FP_conds} as desired.
\end{proof}

\begin{proof}
[Proof of Item~\ref{item:thm_2_item_3} of Theorem~\ref{thm:trivialization_k=1}]

Let $\cuA,\cuB$ be as in the statement and consider 
$\mu\in\cuA \cap \cuV_\mupart$, $\nu \in\cuB \cap \cuV_\nupart(\bR(\mu))$.
Let $\mu$ satisfy
$\schur(\bR(\mu),r_{00}) >  \tau_0  \E_\nu[\ell'(v;v_0,w)^2].$
Suppress the dependence on $\mu$ in the notation for $\bR$.

In what follows, let $(x,x_0) \sim \cN(0,\bR)$.
Let $\nu_{\bR} := \cN(\bzero,\bR) \otimes \P_w$ denote the joint law of $(x,x_0,w)$.
For fixed $w\in\supp(\P_w)$ and any
coupling $\Gamma_w$ of $\cN(\bzero,\bR)$ and $\nu_{\cdot|w}$, 
 we have
\begin{align}
\label{eq:couplings_bound}
&\big\|\E_{\Gamma_w}\left[\ell'(v;v_0,w) (v,v_0) - \ell'(x;x_0,w) (x,x_0) \big|w\right]\big\|_2\\
    &\quad\quad\le
\E_{\Gamma_w}\left[|\ell'(v;v_0,w)|\| (v,v_0) - (x,x_0)\|_2\big| w \right] 
+
    \E_{\Gamma_w}\left[\|(x,x_0)\|_2  |\ell'(v;v_0,w) - \ell'(x;x_0,w)|\big| w \right]\\
     &\quad\quad\le
\big(\E_\nu[\ell'(v;v_0,w)^2|w]^{1/2} + \|\ell'\|_\Lip \Tr(\bR)^{1/2}  \big)\cdot
    \E_{\Gamma_w}\left[\|(v,v_0) - (x,x_0)\|_2^2\big|w\right]^{1/2}.
\end{align}

Now by Talegrand's $T_2$-transportation inequality we have 
\begin{align}
    \E_w\left[\inf_{\Gamma_w}
    \E_{\Gamma_w}\left[\|(v,v_0) - (x,x_0)\|_2^2\big|w\right]\right]
    &\le 
  \E_w\left[  2 \|\bR\|_\op \E_{\nu_{\cdot|w}}\left[\log\left(\frac{\de \nu_{\cdot|w}(v,v_0)}{\varphi_\bR(v,v_0)} \right)\right]  \right]\\
    &=2 \|\bR\|_\op \KL(\nu_{\cdot|w}\| \cN(\bzero,\bR)))
\end{align}
for $\varphi_{\bR}$ the density function for $\cN(0,\bR),$
where expectation over $w$ in the conditional KL-divergence is taken over $\P_w$.
    So defining
    \begin{equation}
 E_{0}(\nu,\bR) :=\big\|
\E_{\nu_{\bR}}[\ell'(x;x_0,w) (x,x_0)]
-
    \E_\nu[\ell'(v;v_0,w) (v,v_0)]   \big\|_2
    \end{equation}
    and
\begin{align}
    A_0(\bR) &:=
      \tau_0^{-1}\, \schur(\bR,r_{00})^{1/2} + \|\ell'\|_{\Lip} \Tr(\bR)^{1/2}\\
      &\ge 
      \E_\nu[\|\ell'(v;v_0,w)\|_2^2]^{1/2} + \|\ell'\|_\Lip \Tr(\bR)^{1/2},
\end{align}
we conclude that
\begin{equation}
\label{eq:KKT_nu_to_KL}
E_{0}(\nu,\bR)
\le  A_0(\bR) 
\,
\Big(
2 \|\bR\|_\op \KL(\nu_{\cdot|w}\| \cN(\bzero,\bR)) \Big)^{1/2}.
\end{equation}

By Theorem~\ref{thm:fin_dim_var_formula_k=1} and Lemma~\ref{lemma:bars_bound_enough}, we have whenever $\alpha \ge \alpha_0$, 
\begin{align}
    \sup_{s\in\cS_{0}(\nu)}&\Phi_{\nupart,\lambda}(\nu,\bR,s)
    \ge 
    \sup_{\tau \in(0,  \tau_0)}\Phi_{\nupart,\lambda}(\nu,\bR,\tau \alpha^{-1})\\
    &\ge 
     -\frac{\tau_0\lambda}{\alpha} 
     - \E\left[\log\left(1 + \frac{\tau_0\ell''(v;v_0,w)}{\alpha}\right)\right]
    +\frac{1}{2\alpha} \log\left(\frac{\tau_0^2 \E[\ell'(v;v_0,w)^2]}{\alpha \, \schur(\bR,r_{00})}\right) 
     + \frac{E_0(\nu,\bR)^2}{2 A_0(\bR)^2  \|\bR\|_\op}\\
     &\ge -\frac{\tau_0\lambda}{\alpha} 
     - \frac{\tau_0}{\alpha}\E_\nu\left[\ell''(v;v_0,w)\right]
    +\frac{1}{2\alpha} \log\left(\frac{\tau_0^2 \E[\ell'(v;v_0,w)^2]}{\alpha \, \schur(\bR,r_{00})}\right) 
     + \frac{E_0(\nu,\bR)^2}{2 A_0(\bR)^2  \|\bR\|_\op}\\
     &\ge -\frac1{\alpha} \rmB_\star(\nu,\bR)  - \frac1\alpha\log(\alpha)
     + \rmG_\star(\bR)
\end{align}
where
\begin{equation}
    \rmB_\star(\nu,\bR) := \tau_0 \big(\E_\nu[\ell''] + \lambda\big)  + \frac12 \log\left(\frac{\schur(\bR,r_{00})}{\tau_0^2 \E_\nu[\ell'^2]}\right)
\end{equation}
and
\begin{equation}
    \rmG_\star(\bR) := 
   \frac{
\big\|\,\E_{(u,u_0)\sim\cN(\bzero,\bR)}\big[(u,u_0)\ell'(u;u_0,w)\big|_{(u,u_0) =(x,x_0)} \big]  
     + \lambda (r_{11},r_{10})
    \big\|_2^2
    }
    {
    2\,\|\bR\|_\op
    \left(\tau_0^{-1}\, \schur(\bR,r_{00})^{1/2} + \|\ell'\|_{\Lip} \Tr(\bR)^{1/2}\right)^2
    },
\end{equation}
where we used 
the condition $\E_{\nu}[\ell'(v,v_0,w)(v,v_0)] = -\lambda (r_{11},r_{10})$ to derive
\begin{equation}
    E_{0}(\nu,\bR) =  \tilde E_0(\bR) := 
\big\|\,\E_{(u,u_0)\sim\cN(\bzero,\bR)}\big[(u,u_0)\ell'(u;u_0,w)\big|_{(u,u_0) =(x,x_0)} \big]  
     + \lambda (r_{11},r_{10})
    \big\|_2^2.
\end{equation}
\end{proof}

\subsection*{Acknowledgments}

This work was supported by the NSF through Award DMS-2031883, the Simons Foundation
through Award 814639 for the Collaboration on the Theoretical Foundations of Deep Learning,
and the NSF Award MFAI-2501597.

\newcommand{\etalchar}[1]{$^{#1}$}
\providecommand{\bysame}{\leavevmode\hbox to3em{\hrulefill}\thinspace}
\providecommand{\MR}{\relax\ifhmode\unskip\space\fi MR }
\providecommand{\MRhref}[2]{%
  \href{http://www.ams.org/mathscinet-getitem?mr=#1}{#2}
}
\providecommand{\href}[2]{#2}

\newpage
\appendix
\section{Deferred proofs}
\subsection{Proof of Lemma~\ref{lemma:G_to_s_star}}
\label{sec:proof_of_lemma_G}

The following lemma outlines some properties of $G_\ST$.
\begin{lemma}[Properties of $G_\ST$]
\label{lemma:properties_G}
Fix $\nu\in\cuP(\R^{3})$ such that
$L_{0}(\nu)= >0$  (recall that
$L_{0}(\nu):= -\Big( 0 \wedge \inf_{(v,v_0,w) \in \supp(\nu)} \ell''(v;v_0,w)\Big)$). 
Then:
\begin{enumerate}
    \item
    \label{item:lemma_G_item_1}
    The function $G_\ST(s)$ has a unique local (hence global) minimizer $s_{\min} :=s_{\min}(\nu)$ in  $(0,L_{0}^{-1}]$.
    \item 
    \label{item:lemma_G_item_2}
    For $z \in (-\infty, -G_\ST(s_{\min })],$ there is a unique solution $\hat s(z) := \hat s(z;\nu)$ to 
    \begin{equation}
        G_\ST(s) = -z,\quad\quad s\in (0,s_{\min}].
    \end{equation}
\item
    \label{item:lemma_G_item_3}
$\hat s(z)$ is analytic on $(-\infty, -G_\ST(s_{\min })),$
and is the unique solution of $G_\ST(s)=  -z$ with  $\hat s(z) = -1/(\alpha z) + o(1/z)$ as $z\to -\infty;$ therefore, it is (up to a scaling) the Stieltjes transform $s_{\star}(z):= s_{\star}(z;\nu)$ of $\mu_{\star}(\nu)$ at $z$.
    
\item 
    \label{item:lemma_G_item_4}
We have \begin{equation}
    \zeta_{\min } := \inf \supp (\mu_{\star}) = - G_\ST(s_{\min }).
\end{equation}
\end{enumerate}
\end{lemma}

\begin{proof}
Let $U_0 := U_0(\nu) = \sup_{\supp(\nu)}\ell''.$
For Item~\textit{\ref{item:lemma_G_item_1}},
note that $G_\ST(s) = 1/(\alpha s) + O(1)$ as $s \to 0$
    and 
    \begin{equation}
        L_0^{-2} G_\ST'(L_0^{-1}) = -\frac1{\alpha} + \int_{-L_0}^{U_0} \left( \frac{x}{L_0 + x }\right)^{2} \ell''_{ \#\nu}(\de x).
    \end{equation}
    Either the latter quantity diverges to $\infty$ implying $G_\ST'(s) >0$ near $L_0^{-1}$, or the latter quantity is finite implying that $G_\ST(L_0^{-1})$ is finite as well. In both cases, we conclude the existence of a local minimum in $(0,L_{0}^{-1}].$
    To prove uniqueness, note that $G_\ST'(s)<0$ near $0$ and  
    \begin{equation}
        s^2 G_\ST'(s) = -\frac1{\alpha} + \int \left(\frac{ s\ell''(v_1,v_0,w) }{ 1 + s \ell''(v_1,v_0,w)}\right)^2  \nu (\de (v_1,v_0,w))
    \end{equation}
 is strictly increasing on $(0, L_{0}^{-1}).$
Item~\textit{\ref{item:lemma_G_item_2}} is now clear since $s_{\min}$ is the global minimizer of $G_\ST$ on $(0,L_{\min}^{-1}].$
   
For Item~\textit{\ref{item:lemma_G_item_3}}, that $\hat s(z)$ is analytic is clear, while $\hat s(z) = -1/(\alpha z) + o(1/z)$ follows from $G_\ST(s) = 1/(\alpha s) +  O(1)$ as $s\to 0.$
   Since $G_\ST(s) = -z$ characterizes the Stilejtes transform outside of the support of $\mu_{\star }$, the claim follows.

Finally, for Item~\textit{\ref{item:lemma_G_item_4}}, we note that
by definition of Stieltjis transform 
\begin{align}
s_{\star}(z;\nu) = \frac1\alpha\int \frac1{x - z} \mu_{\star }(\de x).
\end{align}
In particular $s_{\star}(z;\nu)$ is analytic on a neighborhood of 
$\Omega_{a}:= \{z: \Re(z)\le a, \Im(z)=0 \}$, if and only if
$a<\zeta_{\min}:=\inf\supp(\mu_{\star })$. 
On the other hand, $\hat{s}(z)$ is analytic on a neighborhood of
$\Omega_a$ if and only if $a<-G_\ST(s_{\min})$ (by the 
analytic implicit function theorem).
Therefore by the above, we must have $\zeta_{\min } = -G_\ST(s_{\min})$.
\end{proof}

The next lemma asserts that the Stieltjes transform $s_{\star }(z)$ can be extended by continuity to 
$\zeta_{\min }$, the left edge of the support.
\begin{lemma}[Extension]
Let $\nu\in\cuP(\R^{k+k_0+1})$ be such that
$L_{0}(\nu) >0,$ and recall $\zeta_{\min }$ defined in Lemma~\ref{lemma:properties_G}. Then
 the Stieltjes transform and the logarithmic potential can be extended by continuity  to $\zeta_{\min }$
 so that they are finite at $z=\zeta_{\min }$.
 In particular,
\begin{equation}
   s_{\star}(\zeta_{\min }) :=
   \lim_{\delta\to 0^+} s_{\star }(\zeta_{\min }-\delta)    
   = \frac1\alpha\int \frac{1}{\zeta - \zeta_{\min }} \mu_{\star }(\de\zeta)  \le L_0(\nu)^{-1}
\end{equation}
and 
\begin{equation}
    \lim_{\delta \to 0^+}\int |\log(\zeta - \zeta_{\min } + \delta)|
    \mu_{\star }(\de \zeta)
    =
    \int |\log(\zeta - \zeta_{\min } )|
    \mu_{\star }(\de \zeta) < \infty.
\end{equation}
\end{lemma}
\begin{proof}
For any $\delta>0$, we have
by Item~\textit{\ref{item:lemma_G_item_3}} of Lemma~\ref{lemma:properties_G} that 
\begin{equation}
   \hat s(\zeta_{\min } - \delta) =  s_{\star }(\zeta_{\min } - \delta)  = \frac1\alpha\int \frac1{\zeta - \zeta_{\min } + \delta} \mu_{\star }(\de \zeta)\,.
\end{equation}
By Item~\textit{\ref{item:lemma_G_item_1}} of the same lemma, we have 
$L_{0}(\nu)^{-1} \ge \hat s(\zeta_{\min };\nu)$ so that
\begin{equation}
     L_{0}(\nu)^{-1} \ge \hat s(\zeta_{\min };\nu)=
     \lim_{\delta \to 0+} \hat s(\zeta_{\min } - \delta ;\nu) =  \lim_{\delta\to 0+}  
     \frac1\alpha \int \frac1{\zeta - \zeta_{\min} + \delta} \mu_{\star }(\de \zeta) =
      \frac1\alpha\int \frac1{\zeta - \zeta_{\min }} \mu_{\star }(\de \zeta)\, ,
\end{equation}
where the last equality follows by monotone convergence.

Finiteness of the limit of the logarthmic potential now follows as a consequence, since the support of
 $\nu$, and hence that of $\mu_\star(\nu)$, is upper bounded.
\end{proof}

\begin{proof}[Proof of Lemma~\ref{lemma:G_to_s_star}]
   The lemma is a direct corollary of the previous two lemmas. 
\end{proof}

\subsection{Proof of Lemma~\ref{lemma:variational_form_logdet}}
\label{sec:pf_lemma_variational_form_logdet}

First fix $z < \zeta_{\min}(\nu)$.
Since $z$ is chosen outside of the support of $\mu_\star(\nu)$, it is easy to check that we have enough regularity to exchange the derivatives and integrals in what follows; we do so without commenting further.
First, note that
\begin{equation}
\label{eq:derivative_K}
    \frac{\partial}{\partial s} K_\ST(s,z;\nu) =  - \alpha \left( G_\ST(s ;\nu) + z\right).
\end{equation}
So to prove Item~\textit{1}, 
note that 
\begin{equation}
 \frac{\partial}{\partial z} K_\ST(s_\star(z;\nu),z;\nu) =  
 \left.\frac{\partial}{\partial s} K_\ST(s, z;\nu) \right|_{s = s_\star(z;\nu)} 
 \frac{\partial s_{\star}}{\partial z} (z;\nu) + 
 \left.\frac{\partial}{\partial z} K_\ST(s, z;\nu) \right|_{s = s_\star(z;\nu)} =  - \alpha s_\star(z)
\end{equation}
where the second equality follows from~\eqref{eq:derivative_K} since $G_\ST(s_\star(z;\nu);\nu) = -z$ by 
Lemma~\ref{lemma:properties_G}. Recognizing that $\alpha s_\star(z;\nu)$ is the Stieltjes transform by Lemma~\ref{lemma:properties_G}, we deduce
\begin{equation}
  \frac{\partial}{\partial z} K_\ST(s_\star(z;\nu),z;\nu) =  -\int \frac1{\zeta - z}  \mu_\star(\nu)(\de \zeta) = 
  \frac{\partial}{\partial z} \int\log(\zeta-z) \mu_\star(\nu)(\de \zeta).
\end{equation}
Eq.~\eqref{eq:K_is_log_pot} now follows by evaluating matching the behavior of $K_\ST$ at $z\to-\infty$.
The remaining items of the lemma now follow directly after differentiating $K_\ST$ to the second order and relating this derivative to $G_\ST'$ and applying  Lemma~\ref{lemma:properties_G} once again.

For $z = \zeta_{\min}(\nu)$, we argue by continuity after applying Lemma~\ref{lemma:properties_G} for regularity.
 
\qed
\subsection{Proof of Lemma~\ref{lemma:prox_regularity}}
\label{proof:prox_regularity}
Let $\delta \in (0,1)$ be so that $s = \delta \sfL^{-1}_{\min}.$
Fix $(x_0,w)$ throughout, and for brevity we use $P(x)$ to denote the proximal operator, i.e., we have $P(x) = \argmin_{z \in \R} \left\{  L(z,x)\right\}$
where
\begin{equation}
   L(z; x)  :=  \ell(z;x_0,w) + \frac1{2s}(z - x)^2.
\end{equation}
Therefore
\begin{equation}
    L''(z;x) =  \ell''(z;x_0,w) + \frac1{s}
    \ge  \sfL_{\min} \left(\delta^{-1} - 1\right) >0,
\end{equation}
so that $L(z;x)$ is strongly convex, and hence the proximal operator is well-defined, differentiable in $x$, and satisfies the first order stationarity condition
\begin{equation}
    s\ell'(P(x); x_0,w) + P(x)  = x.
\end{equation}
Implicit differentiation then gives the formula for the derivative which implies
\begin{equation}
    P'(x) = \frac{1}{1 + s \ell''(P(x); x_0,w)} \in \left[ 
    \frac{1}{1 + s \sfL_{\max}} , \frac1{1 - s \sfL_{\min }}
    \right] \subseteq (c, (1-\delta)^{-1}),
\end{equation}
for some constant $c>0$. This implies injectivity and surjectivity of $P$ (surjectivity follows since 
$P(M)\ge P(0)+M/(1+s\sfL_{\max})$ and $P(-M)\le P(0)-M/(1-s\sfL_{\min})$ whence, for any $M_0>0$, there exists 
$M>0$ such that $P([-M,M])\supseteq [-M_0,M_0]$).
\qed

\subsection{Proof of Lemma~\ref{lemma:prox_density}}
\label{proof:prox_density}
As in the proof of Lemma~\ref{lemma:prox_regularity}, we use $P(x)$ for the proximal operator of interest. 

The first claim follows immediately by Lemma~\ref{lemma:prox_regularity}; since $s\ell'(P(g)) + P(g) = g$, 
we can use the change of variables formula to see
that the density of $P(g)$ is given by Eq.~\eqref{eq:PiDensity}.

To prove the second claim, note that the chain rule for KL divergences gives, for any $\nu \ll \cN(\bzero,\bR)$,
\begin{align}
\label{eq:KL_exp_1_k=1}
    \KL(\nu_{\cdot|w} \| \cN(\bzero,\bR))
    &= 
    \KL(\nu_{\cdot|v_0, w} \| \cN(r_{10}r_{00}^{-1}v_0,\schur(\bR,r_{00})) 
    +\KL(\nu_{v_0| w}\| \cN(0,r_{00}))
    + \E_\nu[\log( 1 + s\ell'')],
\end{align}
meanwhile, by the first item,
we have for any $\nu \in\cuV(\bR)$, and $s \in(0,\sfL_{\min}^{-1})$
\begin{align}
\label{eq:KL_exp_2_k=1}
    \KL(\nu_{\cdot|v_0, w} \| \cN(r_{10}r_{00}^{-1}v_0,\schur(\bR,r_{00})) 
    &= 
    \KL(\nu_{\cdot|v_0, w} \| \pi_{\bR,s}(\cdot |v_0,w )) + \E_\nu[\log(1+ s\ell'')]
    - \frac12 \frac{s^2 \E_\nu[\ell'^2]}{\schur(\bR,r_{00})}
    +\lambda s,
\end{align}
Using Eq.~\eqref{eq:KL_exp_1_k=1} and Eq.~\eqref{eq:KL_exp_2_k=1} in the formula for $\Phi_{\nupart}$ of Theorem~\ref{thm:fin_dim_var_formula_k=1} gives the claim.

\qed

\subsection{Proof of Lemma~\ref{lemma:bars_bound_enough}}
\label{proof:bars_bound_enough}
The condition on $\alpha$ gives
\begin{equation}
\label{eq:bars_in_good_interval}
    \frac{\tau_0}{\sqrt{\alpha}} 
    \le 
    \frac{\tau_0}{\sqrt{\alpha_0}}  
    = \frac1{2\, \sfL_{\min}} 
    <\frac1{\sfL_{\min}}.
\end{equation}
Hence, for any $s \in (0,\tau_0\,\alpha^{-1/2})$,
\begin{equation}
    \left(\frac{s \ell''}{1 +s\ell''}\right)^2  \le \begin{cases}
         (s \sfL_{\max})^2 & \mbox{ if } \ell'' \ge 0\, ,\\
         \frac{(s \sfL_{\min})^2}{(1 - s \sfL_{\min})^2}  & \mbox{ if } \ell'' < 0\,.
    \end{cases}
\end{equation}
Hence, we have
\begin{equation}
    \left(\frac{s \ell''}{1 +s\ell''}\right)^2  < \begin{cases}
         (2 s \sfL_{\max})^2 & \mbox{ if } \ell'' \ge 0\, ,\\
         (2 s \sfL_{\min})^2 & \mbox{ if } \ell'' < 0\,.
    \end{cases}
\end{equation}
which then gives
\begin{equation}
    \E\left[\left(\frac{s \ell''}{1 +s\ell''}\right)^2\right]
    < \frac{4\,\tau_0^2}{\alpha} \left(\sfL_{\min}  \vee  \sfL_{\max} \right)^2 
    =\frac1{\alpha}.
\end{equation}
This along with Eq.~\eqref{eq:bars_in_good_interval} proves that
\begin{equation}
    \left(0,\frac{\tau_0}{\sqrt{\alpha}}\right) \subseteq \left\{
    s \in (0, \sfL_{\min}^{-1} ) : \E_\nu\left[ \left(\frac{s\ell''(v;v_0,w)}{1 + s \ell''(v;v_0, w)}\right)^2\right] < \frac1\alpha
    \right\} \subseteq \cS_{0}(\nu)
\end{equation}
for all $\nu$.
\qed

\subsection{Proof of Theorem~\ref{cor:robust_regression}}
\noindent{\emph{Proof of Claim 1.}}
Let us 
re-parameterize 
\begin{equation}
    \bbeta := \btheta -\btheta_0, 
\end{equation}
and apply Theorem~\ref{thm:fin_dim_var_formula} 
with $\btheta$ replaced by $\bbeta$ and $\btheta_0$ set to $0.$
That is, $\cG_n$ now becomes
\begin{equation}
    \cG_n(\sfA_R,\sfa_L) =\left\{ 
    \bbeta \in\R^d : \|\bbeta\|_2< \sfA_R,\; 
    \frac1n\sum_{i=1}^n \ell'(\bx_i^\sT\bbeta - w_i)^2  > \sfa_L^2,\;
    \sigma_{\min} \left(\bJ \bg_0(\bX\bbeta,\bw) \right) > e^{-\overline{o}(n)}
    \right\}.
\end{equation}
We conclude that that for some high-probability event $\Omega_1'$,
\begin{equation}
    \lim_{\delta\to 0} \lim_{n\to\infty} \frac1n \log \E[Z_n(\cuA,\cuB;\sfA_R,\sfa_L)\one_{\Omega_\delta\cap\Omega_1'}] \le -\inf_{\mu\in\cuA \cap\cuV_{\mupart} \cap \cuG_{\mupart}(\sfA_R)}  
    \inf_{\nu\in\cuB \cap \cuV_{\nupart}(r(\mu)) \cap \cuG_\nupart(\sfa_L)}
    \Phi_{\star}(\mu,\nu)
\end{equation}
where
\begin{align}
    \Phi_\star(\mu,\nu) :=\sup_{s\in \cS_{0}(\nu)}\bigg\{ \frac1\alpha \KL(\mu\|\cN(0,r(\mu)))
    &-\E[\log(1 + s \ell''(v-w))] + \frac1{2\alpha} \log\left(\frac{\alpha e s^2\E[\ell'(v-w)^2]}{r(\mu)}\right)
    \nonumber \\
    &+ \KL(\nu_{\cdot|w} \| \cN(0,r(\mu)))\bigg\}.
\end{align}

Recalling Item~\ref{item:thm_2_item_3} of Theorem~\ref{thm:trivialization_k=1}, in this setting we have
\begin{equation}
    \rmB_{\star}(\nu,r) \le
\frac{
\E_\nu[\ell''(v;w)] 
}{2 \sfL_\vee} 
+  \frac12 \log\left( \frac{4\, r \sfL_\vee^2 }{\sfa_L^2}\right) 
\le \frac12 \log\left( \frac{4\, e \, r \sfL_\vee^2 }{\sfa_L^2}\right) 
\end{equation}
and 
\begin{equation}
    \rmG_\star(r)  \ge 
\frac{(\E[\ell''(r^{1/2} Z;w)]  )^2}{18  \sfL_{\vee}^2}
\end{equation}
for all $(\mu,\nu) \in\cuT^c$ with $\mu \in\cuV_\mupart \cap \cuG_\mupart(\sfA_R)$ and $\nu\in\cuV_{\nupart}(r(\mu)) \cap \cuG_\nupart(\sfa_L).$
From the statement of
Theorem~\ref{thm:trivialization_k=1}, we conclude that for any such $\mu,\nu$
\begin{equation}
    \Phi_{\star}(\mu,\nu) \ge  
\inf_{\beta \in [\sfa_L/2\sfL_{\vee}, \sfA_R]}\frac{(\E[\ell''(\beta Z;w)]  )^2}{18  \sfL_{\vee}^2} 
- 
\frac1{2\alpha} \log\left( \frac{4\, e \, \beta^2 \sfL_\vee^2 \alpha^2}{\sfa_L^2 }\right)
\end{equation}
where we used that for all $(\mu,\nu) \in\cuT^c,$ we have $\sfA_R \ge r^{1/2}(\mu) \ge \sfa_L/(2\sfL_{\vee}).$

Recall now the assumption on $\alpha_\star >1$ in the theorem statement. We must have $\Phi_{\star}(\mu,\nu) > 0$ for all $\alpha > \alpha_\star$ and $(\mu,\nu) \in\cuT^c$
with
$\mu \in\cuV_\mupart \cap \cuG_\mupart(\sfA_R)$ and $\nu\in\cuV_{\nupart}(r(\mu)) \cap \cuG_\nupart(\sfa_L).$
Meanwhile, for $\alpha >1  \ge  \sfL_{\min}^2/{\sfL_{\vee}^2}$, Item~\textit{\ref{item:thm_2_item_2}} guarantees that for all $(\mu,\nu)\in\cuT$
with
$\mu \in\cuV_\mupart \cap \cuG_\mupart(\sfA_R)$ and $\nu\in\cuV_{\nupart}(r(\mu)) \cap \cuG_\nupart(\sfa_L),$
we have $\Phi_\star(\mu,\nu) \ge 0$ with equality if and only if $\mu,\nu$ satisfy Definition~\ref{def:opt_FP_conds}.

\noindent{\emph{Proof of Claim 2.}} Assume that $\mu^\opt,\nu^\opt$  are unique, and take 
for $\eps >0$,
\begin{equation}
\cuA_\eps := \{\mu : W_2(\mu ,\mu^\opt) > \eps\},\quad\quad
\cuB_\eps := \{\nu : W_2(\nu ,\nu^\opt) > \eps\}.
\end{equation}
Then
\begin{align}
     \P\big(\exists \;\textrm{local min.}\; \hat\btheta\in\cG_n :  W_{2}(\hmu(\hat\btheta),& \mu^\opt) + W_2(\hnu(\hat\btheta), \nu^\opt) > \eps\big)
     \le  \P\left( \Omega_{\delta}^c \right)  + \P(\Omega_1'^c)\\
     &+\P\left( \{|\cZ_n(\cuA_\eps,\cuB_\eps) \cap\cG_n(\sfA_R,\sfa_L)|>0\}\cap \Omega_{\delta}\cap\Omega_1'\right).
\end{align}
Taking the limit $n,d\to\infty, n/d \to \alpha > \alpha_\star$ followed by $\delta \to 0$ and using the first item gives the claim.

\qed

\subsection{Proof of Corollary~\ref{prop:tukey}}
Recall the definition
   \begin{equation}
       \cZ_{\cE,\cR} := \left\{
      \btheta : \grad \hat R_n(\btheta) = \bzero,\; 
      \grad^2 \hat R_n(\btheta) \succeq \bzero,\; 
        \hat R_n(\btheta) \in \cE,\;
        \|\btheta - \btheta_0\|_2 \in \cR
       \right\}
   \end{equation} 
   for open sets $\cE\subseteq (0,\kappa^2/6)$ and $\cR \subseteq (0,\sfA_R)$.
  The corollary follows directly from the Theorem \ref{thm:fin_dim_var_formula_k=1}
and Theorem~\ref{thm:trivialization_k=1}
  once we show the following lemma.
\begin{lemma}
Let $\ell(t) := \ell_\tuk(t;\kappa)$ be the Tukey loss with cutoff $\kappa>0$.
Assume $w_i$ are i.i.d., and that for some $b_0\in(0,\kappa/2), b_1>0$,
\begin{equation}
   \P_w\left(|w| \in (b_0,\kappa - b_0) \right)  > b_1.
\end{equation}
Then, there exists some $\alpha_1 >1, c>0$ such that for all $\alpha >\alpha_1$
\begin{align}
\lim_{\substack{n,d \to\infty\\ n/d\to\alpha}}\P\bigg(
&\frac1n\sum_{i=1}^n \left((\bx_i^\sT \bbeta )\ell''( \bx_i^\sT\bbeta- w_i) + \ell'(\bx_i^\sT\bbeta -w_i)\right)^2 > c,\\
&\hspace{30mm}\frac1n \sum_{i=1}^n \ell'(\bx_i^\sT\bbeta - w_i )^2 > c\quad\forall \btheta \in\cZ_{\cE,\cR} ,\; \bbeta = \btheta -\btheta_0
\bigg)  = 1.\nonumber
\end{align}
\end{lemma}

\begin{remark}
    We will prove this lemma via a  uniform convergence argument. A 
    better threshold $\alpha_1$ can be obtained via more specialized
    technique (e.g. by Gaussian comparison inequalities). We do not pursue such a more 
    precise estimate here.
\end{remark}

\begin{proof}
We first show that 
\begin{equation}
\label{eq:lb_1}
    \frac1n\sum_{i=1}^n \ell'(\bx_i^\sT(\btheta-\btheta_0) - w_i)^2   > c \quad\quad \forall\btheta \in\cZ_{\cE,\cR}
\end{equation}
with high probability.
For $Z\sim\cN(0,1)>0$ let
\begin{equation}
    p(\sigma) := \P\left(  | \sigma Z - w| \in \Big(\frac{b_0}{2}, \kappa - \frac{b_0}{2} \Big)\right),\quad\quad (Z,w)\sim\cN(0,1) \otimes \P_w.
\end{equation}
It is easy to check that the assumption on $\P_w$ guarantee that
    $\inf_{\sigma \in (0,\sfA_R)}  p(\sigma) > c_0$
for some $c_0>0$.
Now define
\begin{equation}
\hat F(\bbeta) := \frac1n \sum_{i=1}^n
\one_{\{ |\bx_i^\sT\bbeta - w_i| \in (b_0/2, \kappa - b_0/2)\}}\, ,
\quad\quad
    F(\bbeta):= \E \hat F(\bbeta) = p(\|\bbeta\|_2).
\end{equation}
Then, by standard uniform convergence and VC dimension bounds (e.g.~\cite{vershynin2018high} Ex.~8.3.12), we have 
\begin{equation}
    \sup_{\|\bbeta\|_2 \le \sfA_R} \left|\hat F(\bbeta) - F(\bbeta)\right|
    \le C_0(\sfA_R) \left(\sqrt{\frac{d}{n}} + \sqrt{\frac{\log n}{n}}\right)
\end{equation}
with probability at least $1-n^{-1}$ 
for some $C_0>0$ depending on $\sfA_R$. So we conclude that for some $\alpha_2 := \alpha_2(\sfA_R) > 1$, we have  for all $\alpha > \alpha_2$,
\begin{equation}
    \inf_{\|\bbeta\|_2 \le \sfA_R}  \hat F(\bbeta) > c_1
\end{equation}
for some $c_1 >0$ with high probability.
To see that this implies the lower bound in Eq.~\eqref{eq:lb_1}, note that 
\begin{equation}
    \inf_{t \in (b_0/2, \kappa-b_0/2) } \ell'(t)^2 > 0
\end{equation}
for all $b_0< \kappa.$

Via a similar argument, we can then find some $\alpha_3 >1$ and $c_4>0$ so that
$$
\frac1n\sum_{i=1}^n \left((\bx_i^\sT \bbeta )\ell''( \bx_i^\sT\bbeta- w_i) + \ell'(\bx_i^\sT\bbeta -w_i)\right)^2 > c_4
$$ with high probability as $n/d\to \alpha > \alpha_3.$ To see this, note that for some constant $c_5>0$, one can check that 
\begin{equation}
|(\bx_i^\sT\bbeta)\ell''(\bx_i^\sT\bbeta - w_i) + \ell'(\bx_i^\sT\bbeta - w_i)| 
\ge  c_4 \one_{\{ |x_i^\sT\bbeta -w_i| \in \cI\}}
\end{equation}
where
\begin{equation}
\cI := [-\kappa + \eps, v_1 -\eps] \cup \bigcup_{i=1}^4 [v_i - \eps, v_i+\eps]  \cup  [v_4 -\eps, \kappa -\eps]
\end{equation}
for some $v_1,v_2,v_3,v_4 \in [-\kappa,\kappa].$
\end{proof}

\section{Additional details for numerical experiments}
\label{sec:numerical_details}
\paragraph{Parameters for gradient descent.}
In all experiments involving gradient descent, we  initialized  at
 $\hat\btheta_{(0)}\sim\Unif(\S^{d-1}(1))$ and ran $T:=100,000$ iterations of a gradient 
 descent scheme with simple backtracking. Namely, at iteration $t\in[T],$ the step-size 
 is chosen as
\begin{equation}
    \hat\eta_\up{t} :=\max_{j\in\{0,1,\dots,5\}} \left\{ \gamma^j : \hat R_n\big(\hat \btheta_\up{t-1} - \gamma^j \grad \hat R_n(\hat\btheta_\up{t-1})\big) < \hat R_n(\hat\btheta_\up{t-1})\right\},
\end{equation}
for $\gamma = 0.1,$  and we set 
\begin{equation}
\hat\btheta_\up{t} := \hat \btheta_\up{t-1} - \hat\eta_\up{t} \grad \hat R_n(\hat\btheta_\up{t-1})\,.
\end{equation}
The algorithm is terminated when the maximum iteration is reached, or when the $\|\grad\hat R_n(\hat \btheta_\up{t})\|_2 \le 10^{-9}.$
In all the results shown, the algorithm terminated before reaching the maximum number of iterations.

\paragraph{Solving the fix-point equations.}
As mentioned in Section~\ref{sec:Numerical}, we iterate the following numerical scheme to solve the fix-point equations corresponding to the optimal solution $\nu^\opt:$
\begin{align}
\rho_{t+1} &=  \left(\alpha s_t^2 \E[\ell'(v_t-w)^2] \right)^{1/2},\quad\quad
s_{t+1} = \left(\alpha \E\left[\frac{\ell''(v_{t}-w)}{1+s_t \ell''(v_{t}-w)}\right] \right)^{-1} \wedge \sfL_{\min}^{-1},\\
v_t&:= \Prox_{\ell(\cdot-w)}(\rho_t G;s_t)\, ,\;\;\;  G\sim\normal(0,1)\,
\end{align}
where the integration is done numerically.
Note that if $(\rho^\star,s^\star)$ is a fixed point for this system, then 
\begin{equation}
\label{eq:numerical_fixed_point}
    \rho^{\star 2} = \alpha s^{\star 2} \E[\ell'(v-w)^2],\quad \quad
s^\star = \left(\alpha \E\left[\frac{\ell''(v-w)}{1+s^\star\ell''(v-w)}\right] \right)^{-1},\quad\quad
s^\star \in (0,\sfL^{-1}_{\min}),
\end{equation}
where expectation is over 
\begin{equation}
    v := \Prox_{\ell(\cdot-w)}(\rho^\star G; s^\star),\quad (G,w)\sim\cN(0,1) \otimes\P_w.
\end{equation}
Using the stationarity condition of the prox, i.e.,
\begin{equation}
s^\star \ell'(v -w) + \Prox_{\ell(\cdot -w )}(\rho^\star G ; s)  = \rho^\star G
\end{equation}
into the first equation in the display of Eq.~\eqref{eq:numerical_fixed_point}, and performing Gaussian integration by parts shows that 
\begin{equation}
    \E[\ell'(v-w) v] = 0.
\end{equation}
So if $s^\star$ additionally verifies the stability condition in Eq.~\eqref{eq:E3}, then such a system is a solution of the equations in Definition~\ref{def:opt_FP_conds}.

\paragraph{Existence of a local-minimizer in $\Ball(\btheta_0,\sfA_R)$ for large $\alpha.$}
\begin{lemma}
   Working under the assumptions of Item~\ref{item:cor_1_item_2} of Corollary~\ref{prop:tukey}, for any $\sfA_R,\kappa >0$, there exists $\tilde\alpha:=\tilde\alpha(\sfA_R,\kappa)$ so that for $\alpha >\tilde\alpha,$
   we have
   \begin{equation}
       \lim_{n,d\to\infty} \P\left(\exists\;\textrm{local min.}\;\hat\btheta\;\;\textrm{with}\;\;\|\hat\btheta - \btheta_0\|_2 < \sfA_R  \right) = 1.
   \end{equation}
\end{lemma}
\begin{proof}
Let $R(\btheta):= \E[\hat R_n(\btheta)] = \E \ell_{\tuk}(\btheta^\sT\bx - y)$ denote the population risk. Define $F(\beta) := \E[\ell(\beta Z - W)]$ for $(Z,W)\sim \cN(0,1)\otimes\P_w.$
Under the assumptions of the corollary, we have by Gaussian integration by parts that $F'(\beta) \ge c_0(\sfA_R) \beta$ for all $\beta \in[0,\sfA_R],$ for some $c_0 >0$,
whence for $\bDelta \in \Ball^d(\bzero,\sfA_R)$, 
\begin{equation}
 R(\btheta_0 + \bDelta) - R(\btheta_0) = F(\|\bDelta\|_2) - F(0) \ge \frac{c_0}{2} \|\bDelta\|_2^2.
\end{equation}

By a standard symmetrization argument, followed by Lipschitz contraction argument then a Lipschitz Gaussian concentration argument, we obtain that on a high-probability event $\tilde\Omega$, we have
\begin{equation}
\sup_{\|\bDelta\|_2 \le \sfA_R } |\hR_n(\btheta_0 + \bDelta) - R(\btheta_0 + \bDelta)|  \le C_0(\sfA_R,\kappa) \frac1{\sqrt{\alpha}}.
\end{equation}
We thereby conclude that on $\tilde \Omega$, we have for $\bDelta$ with $\|\bDelta\|_2  = \sfA_R,$ 
\begin{equation}
    \hat R_n(\btheta_0 + \bDelta) \ge \hat R_n(\btheta_0) + \frac{c_0 \sfA_R^2}{2} - 2 C_0(\sfA_R,\kappa)\frac1{\sqrt{\alpha}} > \hat R_n(\btheta_0) + c_1(\sfA_R,\kappa)
\end{equation}
for some $c_1>0$, where the last inequality holds for $\alpha$ larger than some $\tilde\alpha(\sfA_R,\kappa)$.
Therefore, on the high-probability event $\tilde\Omega,$ there must exist a local minimizer with $\|\hat\btheta - \btheta_0\|_2 < \sfA_R.$

\end{proof}


\begin{thebibliography}{EKBB{\etalchar{+}}13}

\bibitem[AMS25]{asgari2025local}
Kiana Asgari, Andrea Montanari, and Basil Saeed, \emph{Local minima of the empirical risk in high dimension: General theorems and convex examples}, arXiv preprint arXiv:2502.01953 (2025).

\bibitem[BM12]{BayatiMontanariLASSO}
M.~Bayati and A.~Montanari, \emph{{The LASSO risk for gaussian matrices}}, IEEE Trans. on Inform. Theory \textbf{58} (2012), 1997--2017.

\bibitem[Bol14]{bolthausen2014iterative}
Erwin Bolthausen, \emph{An iterative construction of solutions of the tap equations for the sherrington--kirkpatrick model}, Communications in Mathematical Physics \textbf{325} (2014), no.~1, 333--366.

\bibitem[CFM23]{celentano2023local}
Michael Celentano, Zhou Fan, and Song Mei, \emph{Local convexity of the tap free energy and amp convergence for z 2-synchronization}, The Annals of Statistics \textbf{51} (2023), no.~2, 519--546.

\bibitem[CM22]{celentano2022fundamental}
Michael Celentano and Andrea Montanari, \emph{Fundamental barriers to high-dimensional regression with convex penalties}, The Annals of Statistics \textbf{50} (2022), no.~1, 170--196.

\bibitem[dAT78]{de1978stability}
Jairo~RL de~Almeida and David~J Thouless, \emph{Stability of the sherrington-kirkpatrick solution of a spin glass model}, Journal of Physics A: Mathematical and General \textbf{11} (1978), no.~5, 983--990.

\bibitem[DM16]{donoho2016high}
David Donoho and Andrea Montanari, \emph{High dimensional robust m-estimation: Asymptotic variance via approximate message passing}, Probability Theory and Related Fields \textbf{166} (2016), no.~3, 935--969.

\bibitem[DMM09]{DMM09}
David~L. Donoho, Arian Maleki, and Andrea Montanari, \emph{{Message Passing Algorithms for Compressed Sensing}}, Proceedings of the National Academy of Sciences \textbf{106} (2009), 18914--18919.

\bibitem[EKBB{\etalchar{+}}13]{el2013robust}
Noureddine El~Karoui, Derek Bean, Peter~J Bickel, Chinghway Lim, and Bin Yu, \emph{On robust regression with high-dimensional predictors}, Proceedings of the National Academy of Sciences \textbf{110} (2013), no.~36, 14557--14562.

\bibitem[FMM21]{fan2021tap}
Zhou Fan, Song Mei, and Andrea Montanari, \emph{{TAP Free Energy, Spin Glasses and Variational Inference}}, The Annals of Probability \textbf{49} (2021), no.~1, 1--45.

\bibitem[FN12]{fyodorov2012critical}
Yan~V Fyodorov and Celine Nadal, \emph{{Critical behavior of the number of minima of a random landscape at the glass transition point and the Tracy-Widom distribution}}, Physical review letters \textbf{109} (2012), no.~16, 167203.

\bibitem[HMRT22]{hastie2022surprises}
Trevor Hastie, Andrea Montanari, Saharon Rosset, and Ryan~J Tibshirani, \emph{Surprises in high-dimensional ridgeless least squares interpolation}, The Annals of Statistics \textbf{50} (2022), no.~2, 949--986.

\bibitem[Hub73]{huber1973robust}
Peter~J Huber, \emph{{Robust regression: asymptotics, conjectures and Monte Carlo}}, The annals of statistics (1973), 799--821.

\bibitem[LM19]{lelarge2019fundamental}
Marc Lelarge and L{\'e}o Miolane, \emph{Fundamental limits of symmetric low-rank matrix estimation}, Probability Theory and Related Fields \textbf{173} (2019), no.~3, 859--929.

\bibitem[MAB20]{maillard2020landscape}
Antoine Maillard, G{\'e}rard~Ben Arous, and Giulio Biroli, \emph{Landscape complexity for the empirical risk of generalized linear models}, Mathematical and Scientific Machine Learning, PMLR, 2020, pp.~287--327.

\bibitem[Mam89]{mammen1989asymptotics}
Enno Mammen, \emph{Asymptotics with increasing dimension for robust regression with applications to the bootstrap}, The annals of statistics (1989), 382--400.

\bibitem[MBM18]{mei2018landscape}
Song Mei, Yu~Bai, and Andrea Montanari, \emph{The landscape of empirical risk for nonconvex losses}, The Annals of Statistics \textbf{46} (2018), no.~6A, 2747--2774.

\bibitem[MM21]{miolane2021distribution}
L{\'e}o Miolane and Andrea Montanari, \emph{{The distribution of the Lasso}}, The Annals of Statistics \textbf{49} (2021), no.~4, 2313--2335.

\bibitem[MR15]{montanari2015non}
Andrea Montanari and Emile Richard, \emph{Non-negative principal component analysis: Message passing algorithms and sharp asymptotics}, IEEE Transactions on Information Theory \textbf{62} (2015), no.~3, 1458--1484.

\bibitem[SB95]{silverstein1995empirical}
JW~Silverstein and ZD~Bai, \emph{On the empirical distribution of eigenvalues of a class of large dimensional random matrices}, Journal of Multivariate Analysis \textbf{54} (1995), no.~2, 175--192.

\bibitem[SC95]{silverstein1995analysis}
Jack~W Silverstein and Sang-Il Choi, \emph{Analysis of the limiting spectral distribution of large dimensional random matrices}, Journal of Multivariate Analysis \textbf{54} (1995), no.~2, 295--309.

\bibitem[SW22]{schramm2022computational}
Tselil Schramm and Alexander~S Wein, \emph{Computational barriers to estimation from low-degree polynomials}, The Annals of Statistics \textbf{50} (2022), no.~3, 1833--1858.

\bibitem[TAH18]{thrampoulidis2018precise}
Christos Thrampoulidis, Ehsan Abbasi, and Babak Hassibi, \emph{Precise error analysis of regularized $ m $-estimators in high dimensions}, IEEE Transactions on Information Theory \textbf{64} (2018), no.~8, 5592--5628.

\bibitem[TOH14]{thrampoulidis2014gaussian}
Christos Thrampoulidis, Samet Oymak, and Babak Hassibi, \emph{The gaussian min-max theorem in the presence of convexity}, arXiv preprint arXiv:1408.4837 (2014).

\bibitem[VDGK25]{vilucchio2025asymptotics}
Matteo Vilucchio, Yatin Dandi, Cedric Gerbelot, and Florent Krzakala, \emph{Asymptotics of non-convex generalized linear models in high-dimensions: A proof of the replica formula}, arXiv preprint arXiv:2502.20003 (2025).

\bibitem[VdV98]{van1998asymptotic}
Aad~W Van~der Vaart, \emph{Asymptotic statistics}, vol.~3, Cambridge university press, 1998.

\bibitem[Ver18]{vershynin2018high}
Roman Vershynin, \emph{High-dimensional probability: An introduction with applications in data science}, vol.~47, Cambridge University Press, 2018.

\end{thebibliography}
\end{document}